\documentclass{amsart}

\pdfoutput=1

\usepackage{amsthm,amsmath,amssymb}

\usepackage[utf8]{inputenc}
\usepackage[english]{babel}
\usepackage{enumitem}
\usepackage[pdfencoding=auto]{hyperref}
\usepackage[nameinlink]{cleveref} 
\usepackage[square,numbers,sort]{natbib} 

\usepackage{tikz}

\usepackage{xcolor}
\hypersetup{
    colorlinks,
    linkcolor={blue!50!blue},
    citecolor={blue!50!blue},
    urlcolor={blue!80!black}
}

\calclayout

\theoremstyle{plain}
  \newtheorem{theorem}{\bf Theorem}[section]
  \newtheorem{proposition}[theorem]{\bf Proposition}
  \newtheorem{lemma}[theorem]{\bf Lemma}
  \newtheorem{cor}[theorem]{\bf Corollary}

\theoremstyle{remark}
  \newtheorem{remark}[theorem]{\bf Remark}

  \numberwithin{equation}{section}
  
\DeclareMathOperator*{\essup}{ess \, sup}

  \newcommand{\CC}{\mathbb{C}} 
\newcommand{\RR}{\mathbb{R}} 
\newcommand{\NN}{\mathbb{N}} 

\newcommand{\SP}{\mathbb{S}} 

\newcommand{\norm}[1]{\lVert#1\rVert} 

\newcommand{\jp}[1]{\langle #1 \rangle } 

\newcommand{\pv}{{\it P.V.}} 

\newcommand{\ds}[1]{d\sigma_{#1\eta}} 

\begin{document}

\title[Optimal estimates backscattering]{Optimal estimates for the double dispersion operator in backscattering}
\author{ Cristóbal J. Meroño}

\begin{abstract}
We obtain optimal results in the problem of recovering the singularities of a potential from backscattering data.  
To do this we prove new  estimates for the double dispersion operator of backscattering, the first nonlinear term in the Born series.  
In particular, by measuring the regularity in the Hölder scale, we show that there is a one derivative gain in the integrablity sense for suitably decaying potentials $q\in W^{\beta,2}(\RR^n)$ with $\beta \ge (n-2)/2$.  In the case of radial potentials, we are able to give stronger optimal results  in the Sobolev scale.
\end{abstract}


\maketitle

\vspace{-5mm}

\section{Introduction and main theorems}

In this paper we show that the non-smooth part of a complex potential $q(x)$, $x\in \RR^n$ can be partially recovered from backscattering measurments of  scattering solutions of the  Schrödinger equation with Hamiltonian  $H = -\Delta +q(x)$.  A scattering solution is the response of the equation to  plane waves $e^{ik\theta\cdot x}$, $k\in (0,\infty)$ of direction $\theta \in \SP^{n-1}$. 

The main objective in inverse scattering is to reconstruct the potential $q(x)$ from the far field measurements of the scattering solutions. There are  different ways to do the measurements. In backscattering,  the scattered wave produced by the interaction of the potential with the plane wave $e^{ik\theta\cdot x}$ is measured only in the direction $-\theta$. This means essentially that we consider only the waves that are reflected back by the potential  (the echoes).
It is still an unknown if it is possible to recover the potential completely from the backscattering data. In this paper we center on results of partial recovery of  $q(x)$.  A usual approach is to construct, using the scattering data, the Born approximation of the potential. We will denote it by $q_B(x)$. 
As we shall see in the next section $q_B$ is related to the potential through the Born series expansion,
\begin{equation*} 
 {q}_B  \sim  {q} + \sum_{j=2}^{\infty}{Q_{j}(q)}, 
 \end{equation*}
where $Q_{j}(q)$ are certain multilinear operators  describing the multiple dispersion of waves. We will define explicitly $Q_j(q)$ in the next section.  We  call $Q_{2}$   the double dispersion operator of backscattering. 

It is known that if $q$ is small in certain norms, then  the difference $q-q_B$ is small  in appropriate function spaces (see \cite{er3,stefanov}). But for general potentials it has been  shown that although  $q -q_B$ is not small in general, at least it is smoother than $q$. This implies that the Born approximation contains the main singularities of the potential.  This was originally shown in \cite{PSo} in a different scattering problem (full data scattering) that is analogous to the backscattering problem.

As mentioned in \cite{BFRV13}, this problem is specially well suited to consider complex potentials, since in this case the scattering solutions are defined only for high values of the energy $k^2$. This implies that we can define the Fourier transform of  the Born approximation only for high frequencies, and as a consequence, in general we will consider that $q_B(x)$ is defined modulo a $C^\infty$ function.  Of course, this ambiguity  in the definition of $q_B$  has no effect when considering the regularity of $q-q_B$.

Therefore, a central question is to determine with precision which singularities of $q$ can be  recovered from $q_B$ or from the scattering data.  Essentially we want to determine which is the best $\varepsilon(\beta)>0$, such that for every  $q\in W^{\beta,2}(\RR^n)$, $\beta \ge 0$ we have that $q-q_B \in W^{\alpha,2}(\RR^n)$ for all $\alpha<\beta + \varepsilon(\beta)$ (this would be an $\varepsilon(\beta)^-$ derivative gain). 
 In backscattering there have been a great number of works addressing this problem.  For real potentials  we mention \cite{OPS, Re} in dimension  $2$, and \cite{RV,RRe} for dimensions $2$ and $3$. In \cite{RV}  it is shown that the derivative gain $\varepsilon(\beta)$ is always at least $1/2$ for $n=2,3$. In \cite{BM09}, using a certain  modification of $q_B$ and of the $Q_j$ operators,  they show that it is possible to take  $\varepsilon(\beta) = \min(\beta-(n-3)/2,1)$ for $n\ge3$ odd and $\beta \ge (n-3)/2$. More recently, returning to the case of the Born approximation $q_B$, in \cite{back} it has been proved  that it is possible to take $\varepsilon(\beta) = \min(\beta-(n-3)/2,1)$  for every dimension $n\ge 2$ a and complex $q$ (see \Cref{figure}). 
 
 Apart from the previous works,  which use the Sobolev scale to measure the regularity of $q-q_B$,  in \cite{BFRV13} they  use the Hölder scale. With this approach they are able to obtain  for complex potentials and $n=2$, a whole $1^-$ derivative gain in the integrability sense. This should be the best possible result, as we will explain later. In a different spirit, the recovery of singularities  from backscattering data has been studied also in \cite{GU,esw} without resorting to  the notion of the Born approximation.  Instead, the authors reconstruct the conormal singularities of $q$ from the  scattering data using the time domain approach to scattering.

Returning to the results in the Sobolev scale, by constructing certain radial counterexamples, in \cite{back} it has been shown also that necessarily 
\begin{equation} \label{eq:necessary}
\varepsilon(\beta) \le \min(\beta-(n-4)/2,1).
\end{equation}
 This means  that the fact that $q$ must have an increasing amount of a priori regularity in the scale $W^{\beta,2}(\RR^n)$ as $n$ grows,  is a feature of this problem.   This is reasonable since  the condition $q \in L^r$, $r>n/2$ is necessary for the existence of scattering solutions, as mentioned previously.
 We want to address the regularity gap that still exists  between the known positive results and the necessary condition  \eqref{eq:necessary}.  In this paper, under certain restrictions, we close  the gap that analogously appears in the the regularity estimates of $Q_2(q)$, the worst element in terms of regularity of the Born series expansion. This in turn enables us to improve the recovery of singularities results for $q-q_B$.

In the first result we restrict the range to $\beta>(n-2)/2$, and, as in \cite{BFRV13}, we use  the Holder scale $\Lambda^\alpha(\RR^n)$  to measure the regularity of $Q_2(q)$, instead of the Sobolev scale (see the end of this section for a rigorous definition of the functional spaces).
\begin{theorem} \label{teo:main1}
Let $n\ge 3$ and assume that $q\in W_1^{\beta,2}(\RR^n)$ with $\beta > (n-2)/2$.   Then 
\begin{equation} \label{eq:Q2HolderBound}
\norm{{Q}_2(q)}_{\Lambda^{\alpha}}\le C \norm{q}_{W_1^{\beta,2}}^2, 
\end{equation}
for all  $ \alpha<\beta - (n-2)/2$.

 Moreover, if we assume that $q$ is compactly supported, then  for any $ \alpha<\beta - (n-2)/2$,  we have that $q-q_B \in \Lambda^{\alpha}$ modulo a $C^\infty$ function.	
\end{theorem}
 This theorem extends for dimension $n \ge 3$ the  results  proved in  \cite{BFRV13} in dimension 2. It gives for $Q_2(q)$ and $q-q_B$ a $1^-$ derivative gain in the integrability sense. In fact, by the Sobolev inequality we have that
  \begin{equation} \label{eq:SobolevIneq}
  \norm{f}_{\Lambda^{\gamma-(n-2)/2}} \le \norm{f}_{W^{\gamma+1,2}} ,
   \end{equation}
 if $\gamma>(n-2)/2$. By \eqref{eq:necessary} $\varepsilon(\beta)=1$ is the best possible result in the Sobolev scale when $\beta>(n-2)/2$, and hence, \Cref{teo:main1} gives a weaker version of what is expected to be the best possible result. 
  Actually we are able to prove a slightly better result than \eqref{eq:Q2HolderBound}, an estimate for the Fourier transform of $Q_2(q)$ in  $L^1_\alpha(\RR^n)$ (see \Cref{prop:Q2:L1}).
  
We can also obtain the optimal result in the Sobolev scale, but in exchange we have to assume that the potential is a radial function. Consider a constant $C_0>0$. Let $0 \le \chi(\xi) \le 1$, $\xi\in \RR^n$ be a smooth cut-off function satisfying $\chi(\xi)=1$ if $|\xi|>2C_0$ and $\chi(\xi)=0$ if $|\xi|<C_0$. We define the operator $\widetilde Q_j$ by the relation
\begin{equation} \label{eq:cutoff} 
\widehat{\widetilde{Q}_j(q)}(\xi) := \chi(\xi) \widehat{{Q}_j(q)}(\xi),
\end{equation} 
so that $Q_j(q)$ differs from $\widetilde Q_j(q)$ in a smooth function. 
\begin{theorem} \label{teo:main2}
Let $n \ge 2$ and $\beta \ge \min(0,(n-4)/2)$. Then there is a constant $C>0$ such that the following estimate holds 
\begin{equation} \label{eq:Q2estR}
\norm{\widetilde{Q}_2(q)}_{W^{\alpha,2}} \le C\norm{q}_{W_{1}^{\beta,2}}^2,
\end{equation}
for all $\alpha< \beta + \varepsilon(\beta)$ and every $q \in W^{\beta,2}(\RR^n)$  radial, if and only if
\begin{equation} \label{eq:epsQ2}
 \varepsilon(\beta) =  \min(\beta-(n-4)/2,1).
 \end{equation}
\end{theorem}
The same result has been obtained in the full data case without the restriction for radial potentials. See \cite[Theorem 1.2]{BFRV10} for the sufficient condition and \cite[Theorem 1.3]{fix} for the necessary condition.
As a consequence of the previous theorem, we obtain the following corollary of recovery of singularities for the Born approximation.
\begin{cor} \label{teo:corRadial} Let $n\ge 2$ and let $q \in W^{\beta,2}(\RR^n)$ be a compactly supported and radial function. Then  we have that $q-q_B\in W^{\alpha,2 }(\RR^n)$ if
\begin{equation} \label{eq:rangeCor}
  \alpha< \begin{cases}
 \beta + 2(\beta- (n-3)/2), \hspace{5 mm} if \hspace{5mm} (n-3)/2 \le  \beta < (n-2)/2, \\
  \beta + 1, \hspace{28.5 mm} if \hspace{5mm} (n-2)/2 \le \beta<\infty .
         \end{cases} 
\end{equation}
\end{cor}
See \Cref{figure} for a graphic representation of this results. The previous result  gives a $1^-$ derivative  gain in the range $\beta >(n-2)/2$ which is the best  possible result (except for the limiting case $\alpha=\beta+1$) by \cite[Theorem 1.1]{back}. Unfortunately, this is not the case in the range $(n-4)/2 \le  \beta < (n-2)/2$, since, to get an optimal result, the estimates of the other $Q_j$ operators should be improved too.

 Since the results of recovery of singularities of a potential are  non quantitative in nature, to get \Cref{teo:corRadial} we don't need necessarily  a quantitative   estimate like \eqref{eq:Q2estR}, it is just enough to show that the right  hand side is finite. Then,  instead of asking $q$ to be radial, we can consider potentials which satisfy  the much weaker assumption that there is some radial function  $g\in W^{\beta,2}(\RR^n)$ such that $|\widehat{q} (\xi)| \le \widehat{g}(\xi)$. This yields the following corollary.
\begin{cor} \label{teo:corRadial2} Let $n \ge 2$ and let $q \in W^{\beta,2}(\RR^n)$ be  compactly supported. Assume also that there exists some $g\in W^{\beta,2}(\RR^n)$ such that $|\widehat{q} (\xi)| \le \widehat{g}(\xi)$. Then  we have that $q-q_B\in W^{\alpha,2 }(\RR^n)$ if $\alpha$ and $\beta$ satisfy \eqref{eq:rangeCor}.
\end{cor}

 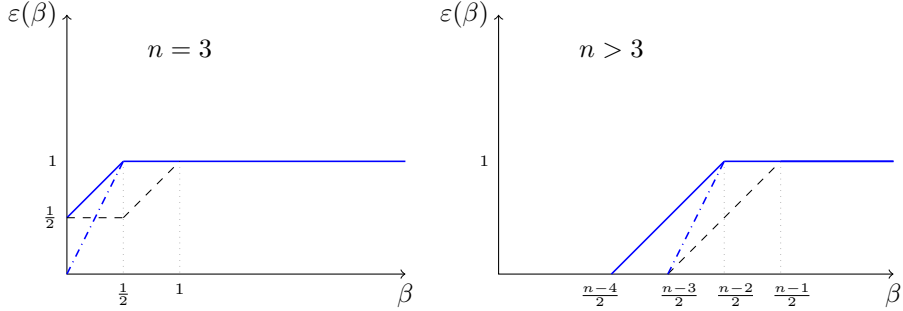
\begin{figure}
\begin{tikzpicture}[scale=1.5]
 \draw[<->] (3,0) node[below]{$\beta$} -- (0,0) --
(0,2.3) node[left]{$\varepsilon(\beta)$};
  \draw [dotted,gray, very thin] (0.5,1) -- (0.5,0);
    \draw [dotted, gray, very thin] (1,1) -- (1,0);
    \draw   [semithick,blue] (0,0.5) -- (0.5,1);
 \draw   [semithick,blue]  (0.5,1) -- (3,1);
   \draw  [dash dot,semithick,blue] (0,0) -- (.5,1);
   \draw [dashed]   (0,.5) -- (.5,.5);
    \draw  [dashed] (0.5,0.5) -- (1,1);
  \node at (1,2) {$n=3$};
 \node [left] at (0,1) { \tiny $1$ };
 \node [left] at (0,0.5) { \tiny $\frac{1}{2}$ };
 \node [below] at (0.5,0) { \tiny $\frac{1}{2}$ };
  \node [below] at (1,0) { \tiny $1$ };
\end{tikzpicture}
\begin{tikzpicture}[scale=1.5]
1. \draw[<->] (3.5,0) node[below]{$\beta$} -- (0,0) --
(0,2.3) node[left]{$\varepsilon(\beta)$};
 \draw [dotted, gray, very thin] (2,1) -- (2,0);
  \draw [dotted,gray,very thin] (2.5,1) -- (2.5,0);
   \draw  [dashed] (1.5,0) -- (2.5,1);
    \draw   [semithick,blue](1,0) -- (2,1);
 \draw  [semithick,blue]  (2,1) -- (3.5,1);   
 \draw  [semithick,blue] (2.5,1) -- (3.5,1);
  \draw  [dash dot,semithick,blue] (1.5,0) -- (2,1);
  \node at (1,2) {$n>3$};
 \node [left] at (0,1) { \tiny $1$ };
 \node [ below] at (1.6,0) { \tiny $\frac{n-3}{2}$ };
  \node [below] at (2.1,0) { \tiny $\frac{n-2}{2}$ };
   \node [below] at (2.6,0) { \tiny $\frac{n-1}{2}$ };
      \node [below] at (.9,0) { \tiny $\frac{n-4}{2}$ };
\end{tikzpicture}
\caption{  
   The solid (blue)  line represents the value of $\varepsilon(\beta)$ for  $Q_2(q)$  given by \Cref{teo:main2}. The dashed line represents the previously known  results of \cite{RRe,back} for general potentials, and the dot dashed (blue) line represents the regularity gain  of $q-q_B$  given  by \eqref{eq:rangeCor}.}
   
  \label{figure}
\end{figure} 

To prove \Cref{teo:main1}  and \Cref{teo:main2}  we estimate the Fourier transform of $Q_2(q)$. In the next section, we will introduce the spherical operator $S_r(q)$, $r\in (0,\infty)$ (see \eqref{eq:Sr})  that involves integrals of $\widehat{q}$ over the so called Ewald spheres. Then,  since we have the formula of \cite{back}
\begin{equation*}
\widehat{Q_2(q)}(\eta) = S_1(q)(\eta) + \pv \int_0^\infty \frac{1}{1-r}S_r(q)(\eta) \, dr,
\end{equation*}
the main task to bound $\widehat{Q_2(q)}$ is to improve the known estimates for $S_r(q)$, and then use the techniques developed in \cite{fix,back} to control the principal value term. To estimate $S_r(q)$ it will be essential to understand certain geometric properties of the Ewald spheres. In the case of $r=1$ we  use a simple case of Santaló's formula in spheres for which we give a short proof in \Cref{sec:santa}.

The recovery of singularities from the Born approximation has been studied in other inverse scattering problems. The case of full data scattering, where all the information in the scattering amplitude is used to construct a Born approximation, has been studied in \cite{PSo,PSe,PSS} for real potentials and in \cite{BFRV10,fix} for complex potentials. Another important  problem not having the radial symmetry of backscattering or full data scattering is the case of fixed angle scattering. In this case the recovery of singularities has been studied in \cite{R} and  \cite{fix}. Surprisingly, it features  the same regularity gap  that appears in backscattering  between the positive and negative results in the Sobolev scale. We mention also that an analogue of the Born approximation has been introduced  to study the recovery of singularities of live loads in Navier elasticity, see \cite{BFPRM1,BFPRM2}.

The question of uniqueness of the inverse scattering problem for backscattering data is still open.  In \cite{RU} it has been proved for $n=3$ that two potentials differing in a finite number of spherical harmonics with radial coefficients must be identical if they have the same backscattering data. The question of uniqueness for small potentials  was studied in \cite{prosser}. Generic uniqueness and uniqueness for small potentials  has  been obtained  in \cite{er3,stefanov}  for dimensions 2 and 3 and in \cite{lagerg}  for $n=3$.  Similar results have been obtained in odd dimension $n\ge 3$ in \cite{Uh01} and  for even dimension in \cite{Wa}.

Let's introduce formally the functional spaces used in the  work. If $\jp{x} = (1+|x|^2)^{1/2}$ and $\alpha \in \RR$, we introduce  the (Bessel) fractional derivative operator $\jp{D}^{\alpha}$ given by the Fourier symbol $\jp{\xi}^\alpha$, and the weighted Sobolev spaces
\[W_\delta^{\alpha,p}(\RR^n) := \{ f\in \mathcal{S}' : \norm{\jp{\cdot}^\delta \jp{D}^{\alpha} f}_{L^p} <\infty\}, \]
for $\delta \in \RR$ and $p\ge1$. We say that $f\in W^{\alpha,p}_{loc}(\RR^n)$ if $\phi f\in W^{\alpha,p}(\RR^n)$ for every $\phi \in C^\infty_c(\RR^n)$, also we usually use the notation $L_\delta^p(\RR^n) := W^{0,p}_\delta(\RR^n)$.

The Hölder spaces $\Lambda^\alpha(\RR^n)$, $\alpha \ge 0$ are the Banach spaces given by  the norm,
\[ \norm{f}_{\Lambda^{\alpha}} = \sum_{|\gamma| < m}\norm{\partial^{\gamma} f}_{\infty} + \sum_{|\gamma| = m} \sup_{t \neq 0} \frac{|\partial^\gamma f(\cdot) - \partial^\gamma f(\cdot -t)|}{|t|^{\sigma}},\]
where we are decomposing $\alpha$ in its integer and fractional parts, $\alpha= m + \sigma$ with $m \in \NN $ and $\sigma \in [0,1)$.

\section{The Born series expansion}

We begin by showing how to construct the scattering solutions, and later we will see how the Born approximation and series are defined. Consider  a scattering solution $u_s(k,\theta,x)$, $k\in(0,\infty)$, $\theta\in \SP^{n-1}$,  of the stationary Schrödinger equation satisfying
\begin{equation} \label{eq.int.1}
\begin{cases}
(-\Delta + q -k^2)u=0 \\
u(x) = e^{i k \theta \cdot x} + u_s(k,\theta,x) \\
\lim_{|x| \to \infty} (\frac{\partial u_s}{\partial r} - iku_s)(x) = o(|x|^{-(n-1)/2}),
\end{cases}
\end{equation}
where the last line is the outgoing Sommerfeld radiation condition (necessary for uniqueness).  If $q$ is compactly supported, a solution $u_s$ of (\ref{eq.int.1}) has the following asymptotic behavior when $|x| \to \infty$
$$u_s(k,\theta,x) =  C |x|^{-(n-1)/2}k^{(n-3)/2} e^{ik|x|} u_\infty(k,\theta,x/|x|) + o(|x|^{-(n-1)/2}) ,$$
 for a certain function $u_\infty(k,\theta,\theta')$, $k\in (0,\infty)$, $\theta,\theta' \in \SP^{n-1}$. As mentioned in the introduction, $u_\infty$ is the so called scattering amplitude or far field pattern, and it is given by the expression
\begin{equation} \label{eq.int.2}
u_\infty (k,\theta,\theta') = \int_{\RR^n} e^{-ik\theta'\cdot y} q(y) u(y) \, dy,
\end{equation}
where it is important to notice that $u$ depends also on $k$ and $\theta$ (for a proof of this fact when $q\in C^\infty_c(\RR^n)$ see for example \cite[p. 53]{notasR}). 

Applying the outgoing resolvent of the Laplacian $R_k$ in the first line of (\ref{eq.int.1}), where 
\begin{equation} \label{eq:resolvent1}
\widehat{R_k(f)}(\xi)  = (-|\xi|^2+k^2 + i0)^{-1}\widehat{f}(\xi), 
\end{equation}
we obtain the Lippmann-Schwinger integral equation
\begin{equation} \label{eq.int.3}
u_s(k,\theta,x) =R_k(q(\cdot) e^{ik\theta \cdot (\cdot)})(x) + R_k(q(\cdot)u_s(k,\theta,\cdot))(x).
\end{equation}

The existence and uniqueness of scattering solutions of (\ref{eq.int.1}) follows from a priori estimates for the resolvent operator $R_k$ and the previous integral equation (\ref{eq.int.3}). In the case of real potentials, this can be shown with  the help of Fredholm theory  for $k >0$, see for example \cite[pp. 79-82]{notasR}. Otherwise, since the norm of the operator $T(f)= R_k(qf)$  decays to zero as $k \to \infty$ in appropriate function spaces, we can also use a Neumann series expansion in (\ref{eq.int.3}) which will be convergent for  $k >k_0$ (in general $k_0 \ge 0$  will depend on some a priori bound of $q$). For our purposes it is enough to consider $q \in L^r(\RR^n)$, $r>n/2$ and compactly supported. Notice that, by the Sobolev embedding, this is satisfied if $q\in W^{\beta,2}(\RR^n)$ with $\beta>(n-4)/2$.  See \cite[p. 511]{BFRV10} for more details and references.

We can introduce now the inverse backscattering problem. If we insert (\ref{eq.int.3}) in (\ref{eq.int.2}),  we can expand  the Lippmann-Schwinger equation in a Neumann series, as we mentioned before.  Then we obtain the Born series expansion relating  the scattering amplitude and the potential in the Fourier transform side.
\begin{align}
\nonumber u_\infty (k,\theta,\theta') = \widehat{q}(\xi) + \sum^l_{j=2}\int_{\RR^n} e^{-ik \theta'\cdot y} (qR_k))^{j-1}(q(\cdot)e^{ik\theta \cdot (\cdot)} )(y) \,dy\\
 \label{eq.int.4} + \int_{\RR^n} e^{-ik \theta'\cdot y} (qR_k)^l(q(\cdot)u_s(k,\theta, \cdot) )(y) \,dy,
\end{align}
where $\xi = k(\theta'-\theta)$ and the last is the error term. Since we are considering complex potentials,  $u_\infty (k,\theta,\theta')$ is not defined for $k \le k_0$ as we have seen. Therefore we also have to ask  $k>k_0$ in (\ref{eq.int.4}). 

The problem of determining $q$ from the knowledge of the scattering amplitude is formally overdetermined in the sense that  the data $u_\infty(k,\theta,\theta')$ is described by $2n-1$ variables, while the unknown potential $q(x)$ has only $n$. We avoid the overdetermination by reducing to the backscattering data, assuming only  knowledge of $u_\infty(k,\theta,-\theta)$, for all $k > k_0$ and $\theta\in \SP^{n-1}$.  There are other possible choices to deal with this difficulty and  that gives rise to the fixed angle scattering problem and the full data scattering problem (see for example \cite{BFRV10,fix}). For backscattering data the problem is formally well determined, and the Born approximation $q_{B}$ is defined by the identity,
\begin{equation}  \label{eq:bornF}
 \widehat{q_{B}}(\xi):= u_\infty(k,\theta,-\theta), \hspace{4mm} \text{where} \hspace{3mm} \xi= -2k\theta.
 \end{equation}
Since in the case of complex potentials $u_\infty(k,\theta,-\theta)$ is not defined for $k\le k_0$, from now on we consider  that $q_B(x)$ is defined modulo a $C^\infty$ function. 

By (\ref{eq:bornF}), the condition  $k>k_0$ is equivalent to asking $|\xi|>2k_0$. Therefore, using the cut-off introduced before (\ref{eq:cutoff}) with $C_0>2k_0$, and assuming convergence of the series, we can write  (\ref{eq.int.4})   as follows
\begin{equation} \label{eq:Bornseries2}
 \chi(\xi)\widehat{q}_B(\xi) = \chi(\xi)\widehat{q}(\xi) + \sum_{j=2}^{\infty}\widehat{ \widetilde Q_{j}(q)}(\xi) ,
 \end{equation}
 where $\widetilde Q$ was defined in  (\ref{eq:cutoff}) and
\begin{equation} \label{eq:Qjraw} 
 \widehat{Q_{j}(q)}(\xi) =\int_{\RR^n} e^{ik \theta\cdot y} (qR_k)^{j-1}(q(\cdot)e^{ik\theta \cdot (\cdot)} )(y) \,dy,
 \end{equation}
again with  $\xi= -2k\theta$. 

The structure of  $\widehat{Q_j(q)}$ has  been recently studied in \cite{back}. 
Let $r\in (0,\infty)$, we introduce the  operator
\begin{equation} \label{eq:Sr}
 S_r(q)(\eta) = \frac{2}{|\eta|(1+r)}\int_{\Gamma_r(\eta)}  \widehat{q}(\xi) \widehat{q}(\eta-\xi) \, d\sigma_{r\eta}(\xi),
 \end{equation}
where $\Gamma_r(\eta)$ is the modified Ewald sphere, 
\begin{equation}  \label{eq:ewald}
 \Gamma_r(\eta) =  \{ \xi\in \RR^n : |\xi-\eta/2| = r|\eta/2|\}.
 \end{equation}
 We call $S_r$ the spherical operator of double dispersion since it involves a spherical integral and a radial parameter $r$.
In \cite[Proposition 3.1]{back} it is shown that the Fourier transform of the double dispersion operator can be decomposed as a sum of a spherical operator and principal value operator $P(q)$,
\begin{equation}  \label{eq:Q2strct}
 \widehat{Q_{2}(q)}(\eta) = S_1(q)(\eta) + P(q)(\eta),
 \end{equation}
where 
\begin{equation*} 
P(q)(\eta) = \pv \int_{0}^{\infty} \frac{1}{1-r} S_r(q)(\eta)   \, dr.
 \end{equation*}
As a consequence,  the main point to bound the double dispersion operator is to estimate $S_r(q)$. Higher order  $Q_j$ operators have a similar, though more complex structure, see \cite[Proposition 5.1]{back}.

We examine now the question of the convergence in Sobolev spaces of the Born series, an essential step to obtain results of recovery of singularities. Taking the inverse Fourier  transform of \eqref{eq:Bornseries2}, we can write, modulo a $C^\infty$ function that 
\begin{equation} \label{eq:bornseries}
q-q_B = \widetilde{Q}_2(q) + \sum_{j=3}^\infty \widetilde{Q}_j(q) .
\end{equation}
In fact,  Theorem 1.3 and Proposition 2.1 of  \cite{back} imply directly the following result.
\begin{proposition} \label{prop:convergence}
Let  $n\ge 2$, $\ell\ge 2$ and $\beta \ge 0$. If $q\in W^{\beta,2}(\RR^n)$  is compactly supported then the series $\sum_{j=\ell}^{\infty} \widetilde Q_j(q),$
converges  in $W^{\alpha,2}(\RR^n)$ provided we take $C_0 = C_0(n,\alpha,\beta,q)$ in $(\ref{eq:cutoff})$  large enough, and that the following condition holds, 
\begin{equation*}
  \alpha< \begin{cases}
 \beta + (\ell-1)(\beta- (n-3)/2), \quad if \quad (n-3)/2 \le  \beta < (n-1)/2, \\
  \beta + (\ell-1), \hspace{27 mm} if \quad (n-1)/2 \le \beta<\infty .
         \end{cases} 
\end{equation*}
\end{proposition}

With this result we can show that to prove \Cref{teo:main1} it is enough to show  estimate \eqref{eq:Q2HolderBound}. In fact if $q\in W^{\beta,2}(\RR^n)$ with $\beta>(n-2)/2$, then $\sum_{j=3}^\infty \widetilde{Q}_j(q)$  converges in $W^{\beta+1,2}(\RR^n)$, so \eqref{eq:SobolevIneq} gives the desired estimate for the series in the $\Lambda^{\alpha}(\RR^n)$ norm. \Cref{teo:corRadial} follows in the same way from \Cref{teo:main2}.

Before going to the next section, we want to highlight the following property of Sobolev norms that we will use frequently in this work.
\begin{remark} \label{remark:Sob}
We have that $W^{\beta,2}_\delta \subset W^{\beta',2}_{\delta'}$ if $\beta \ge \beta'$ and $\delta \ge \delta'$. This follows from the equivalence
\begin{equation*} 
  \norm{\jp{\cdot}^\delta \jp{D}^{\beta} f}_{L^2} \sim \norm{ \jp{D}^{\beta} \jp{\cdot}^\delta f}_{L^2},
  \end{equation*}
  and Plancherel theorem, see for example \cite[Definition 30.2.2]{HormanderIV}. We will also use the inequality
\[
\norm{x_if}_{W^{\beta,2}_\delta} \le C\norm{f}_{W^{\beta,2}_{\delta+1}},
\]  
(this can be proved for example for integer values of $\beta$ and extended by interpolation).
  \end{remark}

\section{ Proof of  Theorem \ref{teo:main1} } \label{sec:SrHolder}

To obtain estimates in the Hölder norm, we will use the inequality
\begin{equation} \label{eq:L1Lambda} \norm{f}_{\Lambda^\alpha} \le   C\norm{\widehat{f}}_{L^1_\alpha} ,
\end{equation}
(see \Cref{prop:L1Lamda} in the Appendix  for a short proof of this fact). By estimate \eqref{eq:L1Lambda}, we have that  \eqref{eq:Q2HolderBound} follows directly from the following proposition.
\begin{proposition} \label{prop:Q2:L1} 
Assume $q\in W^{\beta,2}(\RR^n)$ where $\beta > (n-2)/2$ and $n\ge3$.  Then we have that
\begin{equation*} 
\norm{\widehat{{Q}_2(q)}}_{L^1_\alpha}\le C \norm{q}_{W^{\beta,2}}\norm{q}_{W^{(n-2)/2,2}}, 
\end{equation*}
for all  $ \alpha<\beta - (n-2)/2$.
\end{proposition}
In the introduction we have mentioned that \Cref{teo:main1} is optimal in the sense that it represents a weaker version of what is expected to be the best possible result in the Sobolev scale. For the interested reader we mention that \Cref{prop:Q2:L1} is also optimal (except possibly for the limiting case $\alpha = \beta -(n-2)/2$). This can be verified applying the counterexamples given in \cite[Section 5]{back}, the only necessary change is to use the norm $L^1_\alpha$ instead of $L^2_\alpha$ in the proof of Theorem 1.4 of the same paper.

\begin{proof}[Proof of $\Cref{teo:main1}$]
As mentioned before, the desired estimate for $Q_2(q)$ is obtained from \Cref{prop:Q2:L1} thanks to \eqref{eq:L1Lambda}. The result of recovery of singularities follows then as has been outlined  after \Cref{prop:convergence}.
\end{proof}

\subsection{Estimate of the spherical operator}

 To prove \Cref{prop:Q2:L1} we begin estimating the spherical operator $S_r(q)$. To do that, we need the following result to change the order of integration in the algebraic submanifold of $\RR^{n}\times \RR^n$ defined by  the equation $|\xi-\eta/2|=r^2|\eta/2|$ (recall the definition of $\Gamma_r(\eta)$ given in \eqref{eq:ewald}). We leave the proof for the Appendix, see \Cref{lemma:fubini2}.

\begin{lemma} \label{lemma:fubini}
Let $f \in C_c^\infty(\RR^n)$. Then we have that
\[\int_{\RR^n}\int_{\Gamma_r(\eta)}  f(\eta,\xi)\, \ds{r}(\xi)\, d\eta = \int_{\RR^n} \int_{N_r(\xi)} f(\eta,\xi) \frac{|\eta|}{|\xi|} \, d\sigma_{r,\xi}(\eta) \, d\xi, \]
where we denote by $\sigma_{r,\xi}$ the restriction of the Lebesgue measure to the hypersurface
\begin{equation} \label{eq:Nr}
N_r(\xi) := \{\eta \in \RR^n: |\xi-\eta/2|=r|\eta/2| \}.
\end{equation}
\end{lemma}
 If f $r\neq 1$, $N_r$ is the sphere of center $\frac{2\xi}{1-r^2}$ and radius $\frac{2|\xi|r}{|1-r^2|}$, otherwise for $r=1$ it is  an hyperplane.
We also give the following lemma which is also proved in the  Appendix by direct computation.
\begin{lemma}\label{lemma:integral2}
Let $\SP_\rho\subset\RR^n$ be any sphere of radius $\rho$ and let $\sigma_\rho$ be its Lebesgue measure.
Let  $a,b\ge 0$  satisfy $a+b > n-1$ and $a< n-1$, for all $x\in\RR^n$ we  have that
\begin{equation} \label{eq.intl2} 
 \int_{S_\rho} \frac{1}{|x-y|^a \jp{x-y}^b} \, d\sigma_\rho(y) \le C_{a,b} ,
 \end{equation}
 where the constant $C_{a,b}$ only depends on the parameters $a$ and $b$.
\end{lemma} 
To simplify later computations we define the operator  
\begin{equation}  \label{eq:def:K}
{{K}_{r}(h_1,h_2)}(\eta) := \frac{1}{|\eta|}\int_{\Gamma_{r}(\eta)}|h_1(\xi)| |h_2(\eta-\xi)| \, \ds{r}(\xi).
\end{equation}
We can control the spherical operator if we estimate $K_r$ since
\begin{equation}  \label{eq:KandS}
\left| {S}_r(q)(\eta) \right| \le  \frac{2}{1+r}{{K}_{r}(\widehat{q},\widehat{q})(\eta)}.
\end{equation}
\begin{lemma} \label{lemma:KHolder}
Let $n\ge 3$ and $f_1,f_2 \in W^{\beta,2}(\RR^n)$ with $\beta>(n-2)/2$. Then the estimate
\begin{equation} \label{eq.thm.sph.1}
\norm{K_r(\widehat{f_1},\widehat{f_2})}_{L^1_{\alpha}}\le C  \norm{f_1}_{W^{\beta,2}} \norm{f_2}_{W^{(n-2)/2,2}} + C  \norm{f_2}_{W^{\beta,2}} \norm{f_1}_{W^{(n-2)/2,2}}, 
\end{equation}
 holds when $\alpha<\beta - (n-2)/2$.
\end{lemma}
\begin{proof}
We consider the case $r\neq1$. The case $r=1$ can be proved similarly, though we provide a somewhat more elegant estimate in \Cref{prop:S1Holder}, using a special case of Santaló's formula.

 In the first place, observe that the change of variables $\xi'=\eta-\xi$ leaves invariant the Ewald sphere $\Gamma_r(\eta)$, since it changes a point by its antipodal point on the sphere. We define $\Gamma_r^+(\eta) := \{\xi\in \Gamma_r(\eta):|\xi|\ge|\eta-\xi|\}$, which is the half-sphere at greater distance from the origin. Then, using the mentioned change of variables we can reduce the integrals   over $\Gamma_r(\eta)$ to integrals over $\Gamma_r^+(\eta)$,
\begin{multline*}
\norm{K_r(\widehat{f_1},\widehat{f_2})}_{L^1_{\alpha}}  =  \int_{\RR^n} \frac{\jp{\eta}^\alpha}{|\eta|} \int_{\Gamma_r^+(\eta)} |\widehat{f_1}(\xi) ||\widehat{f_2}(\eta-\xi)| \, \ds{r}(\xi) \, d\eta  \\
+ \int_{\RR^n} \frac{\jp{\eta}^\alpha}{|\eta|} \int_{\Gamma_r^+(\eta)}|\widehat{f_2}(\xi) ||\widehat{f_1}(\eta-\xi)| \, \ds{r}(\xi) \, d\eta.
\end{multline*}
 We are going to estimate only the first term since the estimate of the second follows simply by interchanging the roles of $f_1$ and $f_2$. Let's denote it by $I_1$. 
 
 We define the set
  \[{N_r^+(\xi) := \{\eta \in N_r(\xi): |\xi|\ge|\eta-\xi|\}},\]
   and here we have that $|\eta| \le 2|\xi|$. Consider now $\varepsilon>0$ and  fix $\beta =\alpha+(n-2)/2+2\varepsilon$. Changing the order of integration (\Cref{lemma:fubini}) we obtain
\begin{align}
\nonumber I_1 &\le C\int_{\RR^n} \frac{1}{|\eta|\jp{\eta}^{(n-2)/2+\varepsilon}}\int_{\Gamma_r^+(\eta)}|\widehat{f_1}(\xi)|\jp{\xi}^{\beta-\varepsilon}| \widehat{f_2}(\eta-\xi)| \,\ds{r}(\xi) \, d\eta \\
\nonumber &=   C\int_{\RR^n} |\widehat{f_1}(\xi)|\jp{\xi}^{\beta} \int_{N_r^+(\xi)}\frac{|\widehat{f_2}(\eta-\xi)| }{|\xi|\jp{\xi}^{\varepsilon}\jp{\eta}^{(n-2)/2+\varepsilon}} d\sigma_{r,\xi}(\eta) \, d\xi \\
 \label{eq.thm.spH.2} &\le  C\norm{f_1}_{W^{\beta,2}} \,I_2,
\end{align}
where to get the last inequality  we have applied Cauchy-Schwarz inequality in the $\xi$ variables so that
\[ I_2 := \left( \int_{\RR^n} \left| \int_{N_r^+(\xi)}\frac{1}{|\xi|\jp{\xi}^{\varepsilon}\jp{\eta}^{(n-2)/2+\varepsilon}}| \widehat{f_2}(\eta-\xi)| d\sigma_{r,\xi}(\eta)\right|^2  d\xi \right)^{\frac{1}{2}}.\]
 Now, let us consider the integral in $N^+_r(\xi)$. Taking into account that $\jp{\xi}^{\varepsilon}\, \ge \, \jp{\xi-\eta}^{\varepsilon}$, we multiply and divide by $|\eta-\xi|^{1/2}\jp{\eta-\xi}^{(n-2)/2}$ before using Cauchy-Schwarz inequality in the $\eta$ variable,
\begin{align}
 \nonumber &\left( \int_{N_r^+(\xi)}\frac{1}{\jp{\xi}^{\varepsilon}\jp{\eta}^{(n-2)/2+\varepsilon}}| \widehat{f_2}(\eta-\xi)| d\sigma_{r,\xi}(\eta)\right)^2 \\
 \nonumber &\le   \int_{N_r^+(\xi)}\frac{1}{\jp{\eta}^{n-2+2\varepsilon}} |\widehat{f_2}(\eta-\xi)|^2 |\eta-\xi|\jp{\eta-\xi}^{n-2} d\sigma_{r,\xi}(\eta) \times \dots \\
 \label{eq.thm.spH.3} & \hspace{40mm} \dots \times \int_{N_r(\xi)} \frac{1}{|\eta-\xi|\jp{\eta-\xi}^{n-2 +2\varepsilon}} d\sigma_{r,\xi}(\eta).
\end{align}
But, since  $n \ge 3$,  we can apply  \Cref{lemma:integral2} with $a=1$ and $b=n-2+2\varepsilon$  to get
\begin{equation} \label{eq.thm.spH.4}
\int_{N_r(\xi)} \frac{1}{|\eta-\xi|\jp{\eta-\xi}^{n-2+2\varepsilon}} \, d\sigma_{r,\xi}(\eta) \le C,
\end{equation}
where $C$ does not depend in any way on the sphere $N_r(\xi)$.
If we put together (\ref{eq.thm.spH.3})  and  (\ref{eq.thm.spH.4}), using that $|\eta| \le 2|\xi|$ and changing again the order of integration  we get
\begin{align*}
\nonumber  &I_2 \le \left( \int_{\RR^n}  \int_{N_r^+(\xi)}\frac{1}{|\xi|^{2}\jp{\eta}^{n-2+2\varepsilon}} |\widehat{f_2}(\eta-\xi)|^2 |\eta-\xi|\jp{\eta-\xi}^{n-2} \, d\sigma_{r,\xi}(\eta) \, d\xi \right)^{1/2} \\
\nonumber &\le  C   \left( \int_{\RR^n} \frac{1}{|\eta|^{2}\jp{\eta}^{n-2+2\varepsilon}}   \int_{\Gamma_r(\eta)} |\widehat{f_2}(\eta-\xi)|^2 |\eta-\xi|\jp{\eta-\xi}^{n-2} \, \ds{r}(\xi) \, d\eta \right)^{1/2} \\
   &\le C  \left( \int_{\RR^n} \frac{1}{|\eta|^{2}\jp{\eta}^{n-2+2\varepsilon}}   \int_{\Gamma_r(\eta)} |\widehat{f_2}(\xi')|^2 |\xi'|\jp{\xi'}^{n-2} \, \ds{r}(\xi') \, d\eta \right)^{1/2},
\end{align*}
where we have used in the last line the change of variables $\xi'=\xi-\eta$.
Therefore, if we change the order of integration for the last time, returning to (\ref{eq.thm.spH.2}) we finally obtain
\begin{align*}
I_1 &\le C  \norm{f_1}_{W^{\beta,2}}\left( \int_{\RR^n} |\widehat{f_2}(\xi)|^2 \jp{\xi}^{n-2} \int_{N_r(\xi)}  \frac{1}{|\eta|\jp{\eta}^{n-2+2\varepsilon}}  \, d\sigma_{r,\xi}(\eta) \, d\xi \right)^{1/2} \\
&\le C  \norm{f_1}_{W^{\beta,2}} \norm{f_2}_{W^{(n-2)/2,2}} , 
\end{align*}
where we have applied  \Cref{lemma:integral2} to the integral in $N_r(\xi)$. Then the previous estimate yields
\[\norm{K_r(\widehat{f_1},\widehat{f_2})}_{L^1_{\alpha}}  \le C \left( \norm{f_1}_{W^{\beta,2}} \norm{f_2}_{W^{(n-2)/2,2}} + \norm{f_2}_{W^{\beta,2}} \norm{f_2}_{W^{(n-2)/2,2}}\right),  \]
for $\alpha = \beta-(n-2)/2-2\varepsilon$. Taking $\varepsilon>0$ as small as necessary, we recover the statement of the lemma.
\end{proof}

\subsection{ Estimate of the principal value operator}

\begin{proposition} \label{prop:PVHolder}
Let $n \ge 3$ and  $q \in W^{\beta,2}(\RR^n)$ with $\beta>(n-2)/2$. Then we have that 
\begin{equation} \label{eq:PVHolder}
\norm{{P}(q)}_{L^1_\alpha} \le C \norm{q}_{W_1^{\beta,2}} \norm{q}_{W_1^{(n-2)/2,2}},
\end{equation}
if  $\alpha<\beta-(n-2)/2$.
\end{proposition}
We are going to reduce the proof of this proposition to the following couple of lemmas.
\begin{lemma} \label{lemma:PV}
 Let $q\in \mathcal S(\RR^n)$, $1 \le p <\infty $ and $\phi \in C^\infty(\RR^n)$. Then we have that
\begin{multline*}
\norm{\phi \widehat{P(q)}}_{L^p_\alpha} \le \\ C(\gamma, \varepsilon,\delta) \left(\essup_{r \in (0,\infty)} \left[ (1+r)^{\gamma} \norm{\phi S_r(q)}_{L^p_{\alpha +\varepsilon}} \right] +   \essup_{r \in (1-\delta,1+ \delta)} \norm{\phi \partial_r S_r(q)}_{L^p_{\alpha-1}}\right),
\end{multline*}   
for any $\gamma,\varepsilon >0$.
\end{lemma}
\begin{lemma} Let $q\in \mathcal S(\RR^n)$. Then,
\begin{enumerate}[label=(\alph*)]
\item If $r \in (0,\infty)$, $\beta>(n-2)/2$ and $ \alpha < \beta - (n-2)/2$, we have 
\begin{equation}\label{eq:estSr}
\norm{S_r(q)}_{L^1_{\alpha}(\RR^n)}\le C  \frac{1}{1+r}\norm{q}_{W^{\beta,2}} \norm{q}_{W^{(n-2)/2,2}}, 
\end{equation}
\item For every $\eta \neq 0$, $\partial_r S_r(q)(\eta)$ is smooth in the $r$ variable. Moreover, 
\begin{equation} \label{eq:Kpoint}
  \left | \partial_r  S_{r}(q)(\eta) \right | \le C  K_{r}(\widehat{q},\widehat{q})(\eta) + C |\eta|\sum_{i=1}^n K_{r}(\widehat{x_iq},\widehat{q}),
 \end{equation}
 if $r\in (1-\delta,1-\delta)$ for some $0<\delta<1$ fixed.
 \item Under the same conditions of (a), if we also have $r \in (1+\delta,1-\delta)$, 
 \begin{equation} \label{eq:partialSr}
  \norm{\partial_r  S_{r}(q)}_{L^1_{\alpha-1}} 
   \le C   \norm{q}_{W_1^{\beta,2}}^2.
\end{equation}
\end{enumerate}
\end{lemma}
\begin{proof}
First of all, $(a)$ follows directly from \eqref{eq:KandS} and \Cref{lemma:KHolder}. Actually \eqref{eq:estSr} holds for $q\in W^{\beta,2}(\RR^n)$ and not only for the Schwartz class. $(b)$ is the statement of \cite[Lemma 4.4]{back}, and is proved in the mentioned paper by direct computation. Finally, $(c)$ follows from taking the $L^1_{\alpha-1}$ norm of \eqref{eq:Kpoint}. This yields
\begin{equation*}
  \norm{\partial_r S_{r}(q)}_{L^1_{\alpha-1}} 
   \le C  \norm{ {K}_{r}(\widehat{q},\widehat{q}) }_{L_{\alpha}^1}+  C \sum_{i=1}^{n} \norm{ K_{r}(\widehat{x_iq},\widehat q)}_{ L^1_\alpha},
\end{equation*}
for every $r\in(1-\delta,1+\delta)$. Then we can apply \Cref{lemma:KHolder} directly to the first term with $f_1=q=f_2$, and to the second, with $f_1 =q$ and $f_2(x)=x_iq(x)$,
\begin{equation*}
  \norm{\partial_r S_{r}(q)}_{L^1_{\alpha-1}} 
   \le C  \norm{q}_{W^{\beta,2}}^2 +  C \norm{x_i q}_{W^{\beta,2}}\norm{q}_{W^{(n-2)/2,2}} \le C\norm{q}_{W_1^{\beta,2}}^2,
\end{equation*}
where to get the last line we have used \Cref{remark:Sob}.
\end{proof}

\begin{proof}[Proof of $\Cref{prop:PVHolder}$]
Let $\alpha<\beta -(n-2)/2$. Then we can choose an $\varepsilon = \varepsilon(\beta,\alpha)>0$ such that $\alpha +\varepsilon<\beta -(n-2)/2$. Hence by point $(a)$ of the previous lemma we have  
\begin{equation*}
\norm{S_r(q)}_{L^1_{\alpha +\varepsilon}(\RR^n)}\le C  \frac{1}{1+r}\norm{q}_{W^{\beta,2}} \norm{q}_{W^{(n-2)/2,2}}.
\end{equation*}
This, together with point $(c)$ of the same lemma and \Cref{lemma:PV} with $p=1$,  $\gamma \le 1$ and $\phi =1$ yields estimate \eqref{eq:PVHolder} for $q\in \mathcal S(\RR^n)$. The extension for $q \in W^{\beta,2}(\RR^n)$ follows by standard density arguments ($P(q)$ is bilinear, so it is a slightly different case case from a linear operator, see for example \cite[Lemma 5.3]{back}).
\end{proof}

\begin{proof}[Proof of  $\Cref{lemma:PV}$]
Fix some $\tau  >1$ and $0<\delta<1$, and set
\begin{equation} \label{eq:etadelta} 
\delta_\eta :=  \frac{\delta}{\jp{\eta}^{\tau}}.
\end{equation}
Since $q \in \mathcal S(\RR^n)$, it can be seen that for every $\eta \neq 0 $ fixed, $S_r(q)(\eta)$ is smooth in the $r$ variable. Using that $P.V.\int_{|1-r|<a} \frac{dr}{1-r} = 0$ for any $0<a<\infty$,  we have 
\begin{align}
\nonumber P(q)(\eta) &= \\ 
\nonumber =& \int_{|1-r| \le   \delta_\eta} \frac {S_r(q)-S_1(q)}{1-r}(\eta) \,dr +\int_{\delta_\eta <|1-r|<\delta} \frac{S_r(q)(\eta)}{1-r}dr + \int_{\delta \le |1-r|} \frac{S_r(q)(\eta)}{1-r}\,dr \\
  \label{eq.noPV}  :=& \, P_{A}(q)(\eta) + P_{B}(q)(\eta) + P_{C}(q)(\eta).
\end{align}
By Mikowski's inequality we obtain that
\begin{align} 
\nonumber \|\phi P_{C} (q)\|_{L^p_\alpha} &\leq C \int_{\left\{  \delta<|1-r| \right\}} \frac{\|\phi S_r(q)\|_{L^p_\alpha} }{|1-r|} \, dr  \\
\nonumber &\le  \essup_{r \in (0,\infty)}\left[ (1+r)^{\gamma} \|\phi S_r(q)\|_{L^p_\alpha} \right] \int_{\delta<|1-r|} \frac{1}{(1-r)(1+r)^{\gamma}} \, dr \\ 
\label{eq:PV1} &\le C(\gamma,\delta)   \essup_{r \in (0,\infty)} \left[ (1+r)^{\gamma} \|\phi S_r(q)\|_{L^p_\alpha} \right].
\end{align}
On the other hand,  by the fundamental theorem of calculus we have
\[
 P_{A}(q)(\eta) :=  \int_{|1-r| <  \delta_\eta} \frac {S_r(q)-S_1(q)}{1-r}(\eta) \,dr =  \int_{|1-r| <  \delta_\eta} \int_{0}^1  \partial_s S_{s}(q)(\eta)\big |_{s=s(t)} \, dt \,dr,
\]
for $s(t) = (r-1)t +1$. And then, since $|1-r| <  \delta_\eta$ implies the inequality 
\[ \jp{\eta} \, \le \delta^{1/\tau}|1-r|^{-1/\tau},\]
by Minkowski's inequality we obtain that
\begin{align}
 \nonumber &\norm{ \phi P_{A}(q)}_{ L^p_{\alpha}} \\
 \nonumber  &\le  \left( \int_{\RR^n} \jp{\eta}^{p(\alpha-1)}
\left( \int_{|1-r| <  \delta} \frac{\delta^{1/\tau }}{|1-r|^{1/\tau}}\int_{0}^1 \left| \phi(\eta)\partial_s S_{s}(q)(\eta) \big |_{s=s(t)} \right|   dt \, dr \right)^{p} d\eta \right)^{1/p} \\
\nonumber  &\le  \delta^{1/\tau} \int_{|1-r| <  \delta} \int_0^1 \frac{1}{|1-r|^{1/\tau}}
 \, \big \lVert \phi \partial_s S_{s}(q) {\big|_{s=s(t)}} \big \rVert_{L^p_{\alpha-1}}    dr \, dt  \\
  \label{eq:PV2} &\le  \delta  \frac{2\tau}{\tau-1} \essup_{t \in (0,1)} \big \lVert \phi \partial_s S_{s}(q) {\big|_{s=s(t)}} \big \rVert_{L^p_{\alpha-1}}  =  \delta  \frac{2\tau}{\tau-1} \;  \essup_{r \in (1-\delta,1+\delta)} \norm{ \phi \partial_r  S_{r}(q) }_{L^p_{\alpha-1}}  .
\end{align}
To estimate  the remaining term $P_B(q)$ we need  a dyadic decomposition in the $r$ variable. We set $N(\eta) = - \log_2(\delta \jp{\eta}^{-s})$. Then,
\begin{align*}
P_{B}(q)(\eta) :&=  \int_{\delta_\eta < |1-r| <\delta } \frac{S_r(q)(\eta)}{1-r}  \,dr \\
&= \sum_{0 \le j<N(\eta)}  \int_{\{2^{-(j+1)}<|1-r|<2^{-j}\}} \chi_{\delta_\eta < |1-r| <\delta}(r)  \frac{S_r(q)(\eta)}{1-r}  \,dr .
\end{align*}
 If $j=0,1,...,N(\eta)$, for $\eta$ fixed, the definition of $N(\eta)$ implies that ${2^j \le  \jp{\eta}^{s}}/\delta$, therefore
 \begin{equation}  \label{eq.diadic1}
|P_{B}(q)(\eta)|\le \sum_{j=0}^\infty 2^{j+1} \chi_{(\delta 2^{j},\infty)}(\jp{\eta}^{s})\int_{|1-r|<2^{-j}} |{S_r(q)}(\eta)| \,dr .
\end{equation}
 But observe that in the last line we have an operator of the kind
$${P^{\lambda}(q)}(\eta) := \chi_{(\delta{\lambda^{-1}},\infty)} (\jp{\eta}^{s}) \int_{|1-r|\le \lambda} |{S_r(q)}(\eta)| \, dr,$$
 with $0<\lambda\le 1$. If we take any $\varepsilon>0$, computing its $L^p_{\alpha}$ and applying Minkowski's integral inequality we obtain
\begin{equation}  \label{eq.diadic2}
\norm{ \phi P^{\lambda}(q)}_{ L^p_{\alpha}}  \le {\lambda}^{\varepsilon/s}  \int_{\{|1-r|\le {\lambda} \}} \norm{ \phi S_r(q)}_{ L^p_{\alpha+\varepsilon}} \,dr \le {\lambda}^{1+\varepsilon/s}  \essup_{r\in(0,\infty)} \norm{ \phi S_r(q)}_{ L^p_{\alpha+\varepsilon}},
\end{equation}
where we have used that in the region where the characteristic function does not vanish  $\jp{\eta}^{-\varepsilon} \le \delta^{-\varepsilon/s} \lambda^{\varepsilon/s}$. Hence, taking the $L^p_{\alpha}$ norm of (\ref{eq.diadic1}) and applying estimate (\ref{eq.diadic2}) yields
\begin{multline} \label{eq.PV.3}
 \norm{\phi P_{B}(q)}_{ L^p_{\alpha}} \le 2 \sum^{\infty}_{j=0}  2^{j} \norm{ { \phi P^{2^{-j}}(q)} }_{ L^p_{\alpha}} \\
\le    2\delta^{-\varepsilon/s}  \essup_{r\in(0,\infty)} \norm{ \phi S_r(q)}_{ L^p_{\alpha+\varepsilon}} \sum^{\infty}_{j=0}  2^{-j \varepsilon/s} 
\le    C(\varepsilon,\delta) \essup_{r\in(0,\infty)} \norm{ \phi S_r(q)}_{ L^p_{\alpha+\varepsilon}}.
\end{multline} 
This is enough to conclude the proof, putting together \eqref{eq.noPV}, \eqref{eq:PV1}, \eqref{eq:PV2} and \eqref{eq.PV.3}.
\end{proof}


\section{Santaló's formula and the spherical term} \label{sec:santa}

In this section we give a proof of the estimate of the spherical term $S_r$ for the special case $r=1$. The main tool is Santaló's formula in spheres, which enables us to adapt the arguments of \cite{BFRV13} for dimension $n\ge 3$.  In this section we denote by $\sigma$  the restriction of Lebesgue measure to $\SP^{n-1}$ and $\SP^{n-1}$, independently of the dimension. 
\begin{proposition}[Santaló's formula] \label{prop.santalo}
Let $f$ be a $L^1(\SP^{n-1})$ function and $\theta \in \SP^{n-1}$. Then if we define 
\begin{equation} \label{eq:defnStheta}
 \SP^{n-2}_\theta  = \{ \omega \in \SP^{n-1}: \theta \cdot \omega =0 \},
\end{equation}
we have that
\begin{equation} \label{eq.santalo}
\int_{\SP^{n-1}} \int_{\SP^{n-2}_\theta} f(\omega) \, d\sigma(\omega)\, d\sigma  (\theta) = |\SP^{n-2}| \int_{\SP^{n-1}} f(\theta) \, d\sigma  (\theta).
\end{equation}
\end{proposition}
\begin{proof} 
We define the following positive and bounded functional on  $C(\SP^{n-1})$,
\[F(g) := \int_{\SP^{n-1}} \int_{\SP^{n-2}_\theta} g(\omega) \, d\sigma (\omega)\, d\sigma  (\theta).\]
This means that by the Riesz representation theorem, there exists a Radon  measure $\mu$ on $\SP^{n-1}$ such that
\[F(g) =  \int_{\SP^{n-1}} g(\theta) \, d\mu (\theta). \]
But observe that if $O$ is any orthogonal matrix we have that
\[\int_{\SP^{n-2}_\theta} g(O(\omega)) \, d\sigma (\omega) = \int_{\SP^{n-2}_{O(\theta)}} g(\omega) \, d\sigma (\omega),\]
which in turn implies, integrating  in $\SP^{n-1}$ both sides of the previous equation, that $F$ is invariant under rotations. Therefore, the following property must hold in the measure representation of $F$
\begin{equation} \label{eq.santalo.1}
\int_{\SP^{n-1}} g(O(\theta)) \, d\mu (\theta)=\int_{\SP^{n-1}} g(\theta) \, d\mu (\theta).
\end{equation}
One consequence of this fact is that all balls of the same radius in the sphere must have the same $\mu$-measure, that is, $\mu$ is a uniformly distributed measure on $\SP^{n-1}$. 
This is a very rigid property for Radon measures. In fact, all uniformly distributed Radon measures must be equal up to a scalar factor (see \cite[Proposition 3.1.5]{KP}) which  implies that $\mu$ must be  a multiple of the Lebesgue measure on $\SP^{n-1}$. To determine the constant it is enough to compute $F(1)$.
\end{proof}

Since $r=1$ always in  this section, to simplify notation we will drop the subindex $1$, that is, we write $S(q):=S_1(q)$, $\Gamma(\eta):=\Gamma_1(\eta)$, $N(\xi):=N_1(\xi)$ and analogously for similar cases.

\begin{proposition} \label{prop:S1Holder}
 Let $n\ge 3$ and assume that $q\in W^{\beta,2}(\RR^n)$ with $\beta > (n-2)/2$.   Then we have that
\[ \norm{S(q)}_{L^1_\alpha} \le  C  \norm{q}_{W^{\beta,2}}\norm{q}_{W^{(n-2)/2,2}} ,\]
for all $ \alpha < \beta - (n-2)/2$.
\end{proposition}

\begin{proof}
 As in the proof of \Cref{lemma:KHolder}, by the symmetry in $\xi$ and $\eta-\xi$, we have that 
\begin{align*}
\norm{{S(q)}}_{L^1_\alpha}\le C \int_{\RR^n} \frac{1}{|\eta|} \int_{\Gamma^+(\eta)} |\widehat{q}(\xi)|<\xi>^\alpha |\widehat{q}(\eta-\xi)| \,d\sigma_\eta(\xi) \, d\eta,
\end{align*}
where we have used  that in this region $|\xi| \le |\eta| \le 2|\xi|$. Let's change the order of integration using \Cref{lemma:fubini},
\[\norm{{S(q)}}_{L^1_\alpha} \le C \int_{\RR^n} |\widehat{q}(\xi)| \frac{<\xi>^\alpha}{|\xi|} \int_{N^+(\xi)} |\widehat{q}(\eta-\xi)|   d\sigma_\xi(\eta) d\xi.\]
Now, if we change variables in the second integral by fixing $v = \eta-\xi$ we have
\begin{equation} \label{eq:long1}
\norm{{S(q)}}_{L^1_\alpha} \le C \int_{\RR^n} |\widehat{q}(\xi)| \frac{<\xi>^\alpha}{|\xi|}  \int_{D(\xi)} |\widehat{q}(v)|  d\sigma_\xi(v) d\xi,
\end{equation}
where, since $|\xi- \eta/2|= |\eta|/2 \iff \xi \cdot (\eta-\xi) $, $D(\xi)$ is the disc given by $D(\xi) = \{v \in \RR^n : v \cdot \xi =0, \,|v| \le |\xi| \}$.

If we  write the first integral in \eqref{eq:long1} in spherical coordinates taking  $\xi = r\theta$, by Cauchy-Schwarz inequality we obtain
\begin{align} 
\nonumber \norm{{S(q)}}_{L^1_\alpha} &\le C \int_0^\infty r^{n-2}(1+r^2)^{\alpha/2} \int_{\SP^{n-1}} |\widehat{q}(r\theta)| \int_{D(r\theta)} |\widehat{q}(v)| \, \,  d\sigma_{(r\theta)}(v)\, d\sigma (\theta) dr \\
\label{eq:long2} &\le C \int_0^\infty r^{n-2}(1+r^2)^{\alpha/2} \left ( \int_{\SP^{n-1}} |\widehat{q}(r\theta)|^2 d\sigma (\theta) \right)^\frac{1}{2} F(r) \, dr,
\end{align}
where, using the definition of $\SP^{n-2}_\theta$ given in \eqref{eq:defnStheta}, we have
\begin{align*}
F(r) &= \left( \int_{\SP^{n-1}} \left( \int_{D(r\theta)} |\widehat{q}(v)| \, \,  d\sigma_{(r\theta)}(v)\, \right)^2 d\sigma (\theta)  \right)^{\frac{1}{2}} \\
&= \left( \int_{\SP^{n-1}} \left( \int_{\SP_\theta^{n-2}} \int_0^r s^{n-2} \, |\widehat{q}(s\omega)| \, ds \, d\sigma (\omega)\right)^2  d\sigma (\theta) \right)^{\frac{1}{2}} ,
\end{align*}
Then Hölder inequality and Minkoswski's integral inequality yield
\begin{align*}
F(r)  &\le  C \left( \int_{\SP^{n-1}}\int_{\SP_\theta^{n-2}} \left(  \int_0^r s^{n-2} \, |\widehat{q}(s\omega)| \, ds  \right)^2  d\sigma (\omega) \, d\sigma (\theta) \right)^{\frac{1}{2}} \\
&\le  C  \int_0^r s^{n-2}\left( \int_{\SP^{n-1}}\int_{\SP_\theta^{n-2}}  |\widehat{q}(s\omega)|^2 \, d\sigma (\omega) \, d\sigma (\theta) \right)^{\frac{1}{2}}  ds.
\end{align*}
Now, using Santaló's formula (\ref{eq.santalo}) we have that
\[\int_{\SP^{n-1}}\int_{\SP_\theta^{n-2}}  |\widehat{q}(s\omega)|^2 \, d\sigma (\omega) \, d\sigma (\theta) = |\SP^{n-2}| \int_{\SP^{n-1}}  |\widehat{q}(s\theta)|^2 \, d\sigma (\theta) ,\]
and hence, multiplying and dividing by $\jp{s}^{1/2+\varepsilon}$ and using Cauchy-Schwartz inequality we get
\begin{align*}
F(r)  &\le  C  \left( \int_0^r s^{2(n-2)} \jp{s}^{1+2\varepsilon}  \int_{\SP^{n-1}}  |\widehat{q}(s\theta)|^2 \, d\sigma (\theta) \,ds \right)^{\frac{1}{2}} \times \dots \\
& \hspace{50mm}\dots \times \left(\int_0^\infty \frac{1}{\jp{s'}^{1+2\varepsilon}} \, ds' \right)^\frac{1}{2} \\
&\le C \jp{r}^{\varepsilon} \left( \int_0^r s^{n-1}  \int_{\SP^{n-1}}  |\widehat{q}(s\theta)|^2\jp{s}^{n-2}  \, d\sigma (\theta) ds \right)^{\frac{1}{2}} \\
 &\le C \jp{r}^{\varepsilon}\norm{q}_{W^{(n-2)/2,2}},
\end{align*}
(to get the second line is where we have used implicitly the condition $n \ge 3$, so that the exponent of $s^{n-3}$ is non-negative).
Using the estimate for $F(r)$ in (\ref{eq:long2}), and repeating again exactly the same reasoning to bound the resulting integral, we finally obtain
\begin{align*} 
&\norm{{S(q)}}_{L^1_\alpha} \le C \norm{q}_{W^{(n-2)/2,2}} \int_0^\infty r^{n-2}  \left( \int_{\SP^{n-1}}  |\widehat{q}(r\theta)|^2 \jp{r}^{2\alpha  + 2\varepsilon} \, d\sigma (\theta) \right)^{\frac{1}{2}} dr \\
&\le C \norm{q}_{W^{(n-2)/2,2}}  \left( \int_0^\infty r^{2(n-2)} \jp{r}^{1+2\varepsilon} \int_{\SP^{n-1}}  |\widehat{q}(r\theta)|^2 \jp{r}^{2\alpha  + 2\varepsilon} \, d\sigma (\theta)  \, dr \right)^{\frac{1}{2}} \\
&\le C \norm{q}_{W^{\alpha + (n-2)/2+2\varepsilon,2}} \norm{q}_{W^{(n-2)/2,2}},
\end{align*}
so choosing $\beta = \alpha + (n-2)/2+2\varepsilon$ we obtain the desired result.
\end{proof}

\begin{proof}[Proof of $\Cref{prop:Q2:L1}$]
It follows directly by  \eqref{eq:Q2strct},  using \Cref{prop:PVHolder} and  \Cref{prop:S1Holder}.
\end{proof}


\section{Proof of Theorem \ref{teo:main2}}

In this section we assume that $q$ is a radial function. Since in \Cref{teo:main2} we estimate $\widetilde Q_2(q)$, to simplify notation we define $\widetilde S_r(q) : = \chi S_r(q)$ and $\widetilde K_r(\widehat{f_1},\widehat{f_2}) := \chi K_r(\widehat{f_1},\widehat{f_2})$ as in \eqref{eq:cutoff}. We begin by giving estimates for the spherical operator $\widetilde S_r(q)$ and its radial derivative in $W^{\alpha,2}(\RR^n)$.
As in the case of section \ref{sec:SrHolder}, we estimate first  $\widetilde K_r(\widehat{f_1},\widehat{f_2})$. 

\begin{lemma} \label{lemma:Kradial}
 Let $n\ge 2$ and $f_1,f_2 \in {W}^{2,\beta}(\RR^n)$, and assume that $|\widehat{f_2}(\xi)|$ is a radial function. Then, if $r \in (0,\infty)$, $r\neq 1$ and $\beta_0 =\min(-1/2,(n-7)/4) $ we have that
\begin{equation} \label{eq:Kradial}
\norm{ \widetilde K_r(\widehat{f_1},\widehat{f_2})}_{L^2_{\alpha}(\RR^n)} \le C  \,(1+r)^{1-\gamma} \norm{{f_1}}_{W^{\beta,2}}\norm{{f_2}}_{W^{\beta,2}}, 
\end{equation}
for some $\gamma>0$, possibly depending on $\beta$, if the following condition holds
\begin{equation} \label{eq:primrange}
\begin{cases}
\alpha \le 2\beta - (n-3)/2 \quad if \quad  \beta_0< \beta <  (n-2)/2, \\
\alpha < \beta +1 \hspace{18mm} if \quad (n-2)/2\le  \beta<\infty,
\end{cases}
\end{equation}
\end{lemma}
In  the proof we use the following couple of results about integration on spheres.

Let $h: (0,\infty) \to \CC$ be a measurable function.  Let $x\in \RR^n/\{0\}$  and $b>0$, and consider the functional defined by the expression 
\[F_{x,b}(h) := \int_{\SP_b(x)} h(|z|) \, d\sigma_b(z),\]
where $\sigma_b$ is the Lebesgue measure of $\SP_b(x) \subset \RR^n $, the sphere of radius $b$ and center $x$.
\begin{proposition}  \label{prop:BeltMelin}
There is a measure $\mu_{x,b}$ on the real line, absolutely continuous with respect the Lebesgue measure, such that
\[F_{x,b}(h) = \int_0^\infty h(t)\, d \mu_{x,b}(t) .\]
Moreover, $\mu_{x,b}$ satisfies
\begin{equation} \label{eq:BeltMelin}
\frac{d \mu_{x,b}}{dt} = 2^{3-n}c_{n-1} \chi_{x,b}(t)|x|^{2-n} b t  \left( (|x|+b)^2-t^2 \right)^{(n-3)/2} \left (t^2 -(|x|-b)^2\right)^{(n-3)/2},
\end{equation}
 where $\chi_{x,b}$ is the characteristic function of the interval $\left( \left| |x| -b \right|, |x| + b \right)$ and $c_n = |\SP^{n-1}|$. 
\end{proposition}
This formula is a result of  \cite{Leckband},  a proof can be found also in \cite{BM10}. With this proposition we can prove the following lemma.
\begin{lemma} \label{lemma:Mformula}
Let $h$ as before and $f(x) := h(|x|)$. Then, if $r\ne 1$ we have that
\begin{equation} \label{eq:Mformula}
\int_{N_r(\xi)} f(\eta-\xi) \,  d\sigma_{r,\xi}(\eta) \le C \left(\frac{r}{1+r^2}\right)^{(n-1)/2}\int_{0}^{\infty} h(t) t^{n-2} \,dt.
\end{equation}
\end{lemma}
\begin{proof}
By (\ref{eq:Nr}) we know that
$N_r(\xi)$ is a sphere of center $\frac{2\xi}{1-r^2}$ and radius $ b= \frac{2|\xi|r}{|1-r^2|}$. Since in (\ref{eq:Mformula}) $f$ is valued in  in $\eta - \xi$, we can apply  \Cref{prop:BeltMelin} with $x = \frac{2\xi}{1-r^2} -\xi$,
 \[\int_{N_r(\xi)}f(\eta -\xi) d\sigma_{r, \xi}(\eta)=\int_{\mathbb{S}_b(x)}h(|z|)d \sigma_b(z)=\int_0^\infty h(t)d \mu_{x,b}(t).\]
 
On the other hand, if $t \in (||x|-b|,|x|+b)$   we obtain the inequalities      
$$t^2-(|x|-b)^2 \leq t^2 \;\; \textrm{ and } \;\; (|x|+b)^2-t^2 \leq 4|x|b,$$
and from \eqref{eq:BeltMelin}, since  $|x| = |\xi|\frac{1+r^2}{|1-r^2|}$, we get
\begin{align*}
\frac{d \mu_{x,b}}{dt} &\le C |x|^{2-n + (n-3)/2} b^{1+(n-3)/2}t^{n-2} =C \left( \frac{b}{|x|} \right)^{(n-1)/2}t^{n-2} \\
&\le C \left(  \frac{r}{1+r^2} \right)^{(n-1)/2}t^{n-2},
\end{align*}
which gives the desired result.
\end{proof}

\begin{lemma}\label{lemma:integral}
Let $\SP_\rho\subset\RR^n$ be any sphere of radius $\rho$ and let $\sigma_\rho$ be its Lebesgue measure.
Then for any $0< \lambda\le (n-1)/2$, we have that
\begin{equation*} 
\int_{\SP_\rho} \frac{1}{|x-y|^{(n-1)-2\lambda}} \, d\sigma_\rho(y) \le C_\lambda \rho^{2\lambda},
\end{equation*}
 for any $x\in\RR^n$, and for a constant $C_\lambda$ that only depends on  $\lambda$. 
\end{lemma}
The proof of this lemma is very similar to the proof of \Cref{lemma:integral2}.  For the detailed computations see \cite[Lemma 3.3]{fix}. 
\begin{proof}[Proof of  \Cref{lemma:Kradial}]
Since ${\chi}(\eta)=0$  for $|\eta| \le 1$ (see the definition of $\chi$ just before \eqref{eq:cutoff}),  $ \jp{\eta} \le  2|\eta| $ in the region where $\chi$ does not vanish. Then
\begin{multline*}
\norm{ \widetilde K_r(\widehat{f_1},\widehat{f_2})}_{L^2_{\alpha}}  \le    C \left( \int_{\RR^n} |\eta|^{2\alpha-2} \left( \int_{\Gamma_r^+(\eta)} |\widehat{f_1}(\xi) ||\widehat{f_2}(\eta-\xi)| \, \ds{r}(\xi) \right)^2 d\eta \right)^{1/2} \\
  +   C \left( \int_{\RR^n} |\eta|^{2\alpha-2} \left( \int_{\Gamma_r^-(\eta)} |\widehat{f_1}(\xi) ||\widehat{f_2}(\eta-\xi)| \, \ds{r}(\xi) \right)^2 d\eta \right)^{1/2} := I_1 +I_2,
\end{multline*}
where $\Gamma_r^-(\eta) := \{\xi\in \Gamma(\eta):|\xi|<|\eta-\xi|\}$ is the complementary  of $\Gamma_r^+(\eta)$ (introduced in \Cref{lemma:KHolder}). We begin with the estimate of $I_1$.  

Consider a parameter $0<\lambda \le (n-1)/2$. By Cauchy-Schwarz inequality and  \Cref{lemma:integral} we have that
\begin{align*}
 \nonumber I_1^2 &\le C\int_{\RR^n}  |\eta|^{2\alpha-2} \int_{\Gamma_r^+(\eta)} |\widehat{f_1}(\xi) |^2|\widehat{f_2}(\eta-\xi)|^2|\eta - \xi|^{n-1-2\lambda}\, \ds{r}(\xi)\times \dots \\ 
 \nonumber  & \hspace{60mm}\dots \times \int_{\Gamma_r(\eta)} \frac{1}{|\eta-\xi|^{n-1 -2\lambda}} \, \ds{r}(\xi) \,d\eta \\
  &\le Cr^{2\lambda}\int_{\RR^n}  |\eta|^{2\alpha-2 + 2\lambda} \int_{\Gamma_r^+(\eta)} |\widehat{f_1}(\xi) |^2|\widehat{f_2}(\eta-\xi)|^2|\eta - \xi|^{n-1-2\lambda}\, \ds{r}(\xi).
\end{align*}
Then, using that $|\eta| \le 2|\xi|$ in $\Gamma_r^+(\eta)$ and  \Cref{lemma:fubini} to change the order of integration, yields
\begin{equation} \label{eq:I1est}
  I_1^2 \le   C r^{2\lambda}\int_{\RR^n}  |\widehat{f_1}(\xi) |^2 |\xi|^{2\alpha-2 +2\lambda}  \int_{N_r(\xi)}  |\widehat{f_2}(\eta-\xi)|^2|\eta - \xi|^{n-1-2\lambda} \, d\sigma_{r,\xi}(\eta)\, d\xi,
\end{equation}
 (notice that we need $2\alpha-1 +2\lambda \ge 0$).  From now on we fix $\lambda$ such that
\begin{equation} \label{eq:beta}
  \beta = \alpha-1+\lambda.
 \end{equation}
Since $|\widehat{f_2}(\xi)|$ is a radial function, we can write that $|\widehat{f}_2(\xi)| = g(|\xi|)$  for an appropriate function $g$. Then we can apply \Cref{lemma:Mformula} with  $h(t) = g(t)^2 t^{n-1-2\lambda}$ to the second integral of  \eqref{eq:I1est}, and this yields
\begin{align*}
I_1^2 &\le C r^{2\lambda}\left(\frac{r}{1+r^2}\right)^{(n-1)/2} \int_{\RR^n}  |\widehat{f_1}(\xi) |^2 |\xi|^{2\beta} \int_0^\infty  g(t)^2 t^{n-2-2\lambda} t^{n-1} \, dt \, d\xi \\
&\le  C r^{2\lambda}\left(\frac{r}{1+r^2}\right)^{(n-1)/2} \norm{{f_1}}_{{W}^{\beta,2}}^2\norm{{f_2}}_{{W}^{(n-2)/2 - \lambda,2}}^2.
\end{align*}
Analogously for $I_2$, by Cauchy-Schwarz inequality and  \Cref{lemma:integral} we have
\begin{align*}
 I_2^2 &\le C\int_{\RR^n}    |\eta|^{2\alpha-2} \int_{\Gamma_r^-(\eta)} |\widehat{f_1}(\xi) |^2|\xi|^{n-1-2\lambda}|\widehat{f_2}(\eta-\xi)|^2\, \ds{r}(\xi)\times \dots \\ & \hspace{60mm}\dots \times \int_{\Gamma_r(\eta)} \frac{1}{|\xi|^{n-1 -2\lambda}} \, \ds{r}(\xi) \,d\eta\\
 &\le Cr^{2\lambda} \int_{\RR^n}    |\eta|^{2\alpha-2 +2\lambda} \int_{\Gamma_r^-(\eta)} |\widehat{f_1}(\xi) |^2|\xi|^{n-1-2\lambda}|\widehat{f_2}(\eta-\xi)|^2\, \ds{r}(\xi).
\end{align*}
Then, changing the order of integration (\Cref{lemma:fubini}) and using that $|\eta| \le 2|\eta-\xi|$ in $\Gamma_r^-(\eta)$ gives
\[ I_2 \le  Cr^{2\lambda}\int_{\RR^n}  |\widehat{f_1}(\xi) |^2|\xi|^{n-2-2\lambda} \int_{N_r(\xi)}  |\eta-\xi|^{2\alpha-1 +2\lambda} |\widehat{f_2}(\eta-\xi)|^2 \, d\sigma_{r,\xi}(\eta)\, d\xi. \]

 Therefore  we can apply again  \Cref{lemma:Mformula}, this time with  $h(t) = g(t)^2t^{2\alpha-1 +2\lambda}$ and use \eqref{eq:beta} to get
\begin{align*}
I_2^2 &\le C r^{2\lambda}\left(\frac{r}{1+r^2}\right)^{(n-1)/2} \int_{\RR^n}  |\widehat{f_1}(\xi) |^2 |\xi|^{n-2-2\lambda} \int_0^\infty  |g(t)|^2 t^{2\beta} t^{n-1} \, dt \, d\xi \\
&\le  C r^{2\lambda}\left(\frac{r}{1+r^2}\right)^{(n-1)/2} \norm{{f_1}}_{{W}^{(n-2)/2 - \lambda,2}}^2 \norm{{f_2}}_{{W}^{\beta,2}}^2.
\end{align*}
Hence, putting together the estimates of $I_1$ and $I_2$ yields
\begin{multline} \label{eq:KradiaLast}
\norm{ \widetilde K_r(\widehat{f_1},\widehat{f_2})}_{L^2_{\alpha}}  \le Cr^{\lambda} \left(\frac{r}{1+r^2}\right)^{{(n-1)}/{4}}   \norm{{f_1}}_{{W}^{\beta,2}}\norm{{f_2}}_{W^{(n-2)/2 - \lambda,2}} \\
  + \norm{{f_1}}_{W^{(n-2)/2 - \lambda,2}}\norm{{f_2}}_{W^{\beta,2}}.
\end{multline}

Since we want to have the bound $r^{\lambda} \left(\frac{r}{1+r^2} \right)^{(n-1)/4}\le (1+r)^{1-\gamma}$ for some $\gamma>0$, we need to ask $\lambda-(n-1)/4<1$ and hence we need $\lambda<(n+3)/4$.

By \eqref{eq:beta},  the condition $2\alpha -1+ 2\lambda \ge 0$ used in the proof implies we must have $\beta  \ge -1/2$.

  Then, equation \eqref{eq:Kradial} follows directly from \eqref{eq:KradiaLast} in the range $\beta \ge (n-2)/2$. But, together with \eqref{eq:beta}, the restrictions imposed on $\lambda$ yield
\begin{equation}\label{eq:restriccion_1}
\begin{cases}
0 < \lambda < \frac{n+3}{4} \\ 
0 < \lambda \le \frac{n-1}{2}
\end{cases}   \Longleftrightarrow
 \begin{cases}
\beta+1- \frac{n+3}{4}< \alpha < \beta +1  \\ 
\beta +1 - \frac{n-1}{2} \le \alpha < \beta +1.
\end{cases}                
\end{equation}
We can discard the lower bounds for  $\alpha$ using that $\norm{f}_{L^2_\alpha} \le \norm{f}_{L^2_{\alpha'}}$  always holds if $\alpha\le \alpha'$. Therefore   only the restriction $ \alpha<\beta +1$ remains.

Otherwise, if $\beta$ is in the range $0 \le \beta < (n-2)/2 $, estimate (\ref{eq.thm.sph.1}) will follow if we add the extra condition  
\begin{equation} \label{eq:rest2}
(n-2)/2-\lambda \le \beta.
\end{equation}
Then, we have that $\beta \ge \min(-1/2,(n-7)/4)$ by the conditions on $\lambda$ given in the left hand side of \eqref{eq:restriccion_1}. Also, putting together (\ref{eq:beta}) and (\ref{eq:rest2}) we get $\alpha \le 2\beta-(n-4)/2$, which is a stronger condition than $\alpha<\beta+1$ since we are in the range $\beta<(n-2)/2$. Hence, we have obtained the ranges of parameters given in the statement.
\end{proof}

We can estimate now the principal value term.

\begin{proposition} \label{prop:PVSobolev}
Let $n \ge 2$ and  $q \in  W^{\beta,2}(\RR^n)$  be a radial function. Then we have that 
\[\norm{\widetilde{P}(q)}_{W^{\alpha,2}} \le C \norm{q}_{W^{\beta,2}_1}^2,\]
if  the following condition holds
\begin{equation} \label{eq:rangePradial}
  \alpha< \begin{cases}
 2\beta - (n-4)/2, \hspace{5 mm} if \hspace{5mm} \beta_0 \le  \beta < (n-2)/2, \\
  \beta + 1, \hspace{19.5mm} if \hspace{5mm} (n-2)/2 \le \beta<\infty .
         \end{cases} 
\end{equation}
\end{proposition}
\begin{proof}

 Let $n\ge 2$ and assume that $q\in \mathcal S(\RR^n) $ is a radial function.  Let's multiply \eqref{eq:KandS} by $\chi(\eta)$ and apply  \Cref{lemma:Kradial}. Then, if $r\in (0,\infty)$ and $r\neq 1$, for each $\alpha$ in the range \eqref{eq:rangePradial} we can choose an $\varepsilon=\varepsilon(\alpha,\beta) >0$  such that $\alpha +\varepsilon$ is smaller than the left hand side of  \eqref{eq:primrange}. Hence \Cref{lemma:Kradial} gives the estimate
\begin{equation} \label{eq:SrRadial}
\norm{ \widetilde S_r(q)}_{L^2_{\alpha + \varepsilon}}\le C  (1+r)^{-\gamma} \norm{q}_{W^{\beta,2}}^2, 
\end{equation}
for some $\gamma>0$, which can depend on $\beta$.

 Also, multiplying \eqref{eq:Kpoint} by $\chi(\eta)$ and  taking the $L^{2}_{\alpha-1}$ norm we get
\begin{equation} \label{eq:Kpoint2}
  \norm{\partial_r \widetilde S_{r}(q)}_{L^2_{\alpha-1}} 
   \le C  \norm{ \widetilde{K}_{r}(\widehat{q},\widehat{q}) }_{L_{\alpha-1}^2}+  C \sum_{i=1}^{n} \norm{\widetilde K_{r}(\widehat{x_iq},\widehat q)}_{ L^2_\alpha},
\end{equation}
assuming that $r \in(1+\delta,1-\delta)$,  for some  $0<\delta<1$ fixed.  Then, if we also consider $r\neq 1$, we can apply again \Cref{lemma:Kradial} to the first term on the right hand side with $f_1 =f_2= q$, and to the remaining terms with $f_2 = q$ to obtain
\begin{equation*} 
 \norm{\widetilde K_{r}(\widehat{x_iq},\widehat q)}_{ L^2_\alpha}\le C  \norm{x_iq}_{W^{\beta,2}} \norm{q}_{W^{\beta,2}} \le \norm{q}_{W_1^{\beta,2}} \norm{q}_{W^{\beta,2}},
\end{equation*}
in the range \eqref{eq:rangePradial} (we have used again \Cref{remark:Sob}). This yields 
\begin{equation}  \label{eq:Srpartial}
\norm{  \partial_r \widetilde{S}_{r}(q)}_{ L^2_{\alpha-1}}\le C  { \norm{q}_{W_1^{\beta,2}}^2},
\end{equation}
for $r \in(1+\delta,1-\delta)$, $r\neq 1$. Hence, we can apply  \Cref{lemma:PV} with $p=2$ and $\phi(\eta) = \chi(\eta)$ to obtain the desired estimate for $\widetilde P(q)$ and q in the Schwartz class. The extension for $q \in  W^{\beta,2}(\RR^n)$ follows by a density argument as we mentioned in the proof of \Cref{prop:PVHolder}.
\end{proof}

Since  \Cref{lemma:Kradial} does not include the case $r=1$, to estimate  $S(q)$ we need to study separately this case  (we remind that $S(q) = S_{1}(q)$ as we defined in section \ref{sec:santa}).

\begin{proposition}  \label{prop:S1radial}
Let $n \ge 2$ and  $q \in  W^{\beta,2}(\RR^n)$ be a radial function. Then we have that 
\[\norm{ \widetilde S(q)}_{W^{\alpha,2}} \le C \norm{q}_{W^{\beta,2}}^2,\]
in the range \eqref{eq:rangePradial}.
\end{proposition}
\begin{proof}
We can reason exactly in the same way we did in \Cref{lemma:Kradial} to arrive to equation \eqref{eq:I1est}. Since now $r=1$, this gives
\begin{equation} \label{eq:SradialLast}
 \norm{\widetilde{S}(q)}_{L^2_\alpha}
 \le  C \frac{}{} \int_{\RR^n}  |\widehat{q}(\xi) |^2 |\xi|^{2\alpha-2 +2\lambda}  \int_{N(\xi)}  |\widehat{q}(\eta-\xi)|^2|\eta - \xi|^{n-1-2\lambda} \, d\sigma_\xi(\eta)\, d\xi,
\end{equation}
the only difference is that $N(\xi)$ is now an hyperplane and not a sphere. As in \eqref{eq:long2}, in the second integral we introduce the change  of variables $v=\eta-\xi$ which  translates the  $N(\xi)$ to the origin. Then we can take polar coordinates $v=s\theta$ in the resulting hyperplane. With a slight abuse of notation we can write that $\widehat{q}(s\theta) = \widehat{q}(s)$, since $\widehat{q}$  is radial.   This yields 
\begin{align*}
 \int_{N(\xi)}  |\widehat{q}(\eta-\xi)|^2|\eta - \xi|^{n-1-2\lambda} \, d\lambda_\xi(\eta)\, &= \int_0^{|\xi|}   |\widehat{q}(s)|^2s^{n-1-2\lambda} s^{n-2} \,  ds \\
 &= \int_{|v| \le |\xi|} |\widehat{q}(v)|^2|v|^{n-2-2\lambda} \, dv,
\end{align*}
and hence
\begin{multline*}
 \norm{\widetilde{S}(q)}_{L^2_\alpha}
 \le  C \frac{}{} \int_{\RR^n}  |\widehat{q}(\xi) |^2 |\xi|^{2\alpha-2 +2\lambda} \,  d\xi  \int_{\RR^n} |\widehat{q}(v)|^2|v|^{n-2-2\lambda} \, dv \\
 \le C \norm{q}_{W^{\beta,2}} \norm{q}_{W^{(n-2)/2 -\lambda,2}}.
\end{multline*}
Using \eqref{eq:beta} and choosing the parameters as in the final part of \Cref{lemma:Kradial} we get the desired result.
\end{proof}

\begin{proof}[Proof of $\Cref{teo:main2}$]
Multiplying  \eqref{eq:Q2strct} by the cut-off $\chi(\eta)$ we get
\begin{equation} \label{eq:Q2strct2}
 \widehat{\widetilde{Q}_{2}(q)}(\eta) = \widetilde{S}(q)(\eta) + \widetilde{P}(q)(\eta).
\end{equation}
Then \Cref{prop:PVSobolev} and  \Cref{prop:S1radial} together with  Plancherel theorem give the desired estimate for $\widetilde{Q}_2(q)$. As mentioned in the introduction, the necessary condition for $\varepsilon(\beta)$ was proved in \cite[Theorem 1.4]{back}.  
\end{proof} 

\begin{proof}[Proof of $\Cref{teo:corRadial2}$]
By \Cref{prop:convergence} and \eqref{eq:bornseries} it is enough to show that 
\[
\norm{\widetilde{Q}_2(q)}_{W^{\alpha,2}} < \infty ,
\]
 for $\alpha<\beta + \varepsilon(\beta)$ and $\varepsilon(\beta)$ given by \eqref{eq:epsQ2}. We sketch the main ideas of the proof.

Observe that in \Cref{lemma:Kradial}, we have used only that $|\widehat{f_2}(\xi)|$ is a radial function. By \eqref{eq:def:K} and the assumption that $|\widehat{q} |\le \widehat{g}$, we have
\[
\widetilde{K}_r(\widehat{f_1},\widehat{q}) \le \widetilde{K}_r(\widehat{f_1},\widehat{g}).
\]
Hence applying \Cref{lemma:Kradial} to the right hand side with $f_2 = g$ yields
$\norm{\widetilde{K}_r(\widehat{f_1},\widehat{q})}_{L^2_\alpha} < \infty$. This estimate can be used to show that the $L^2_\alpha$ norms of $\widetilde S_r(q)$ and $\widetilde \partial_r S_r(q)$  are finite in the desired range of $\alpha$, exactly with the same method used to obtain \eqref{eq:SrRadial} and \eqref{eq:Srpartial}. Then, as we have shown in the proof of \Cref{prop:PVSobolev}, we get $\norm{\widetilde{P}(q)}_{L^2_\alpha} < \infty$. The case of $S(q)$ much simpler, just observe that, since we are assuming $|\widehat{q} |\le \widehat{g}$ , \eqref{eq:SradialLast} can be replaced by
\[
 \norm{\widetilde{S}(q)}_{L^2_\alpha}
 \le  C \frac{}{} \int_{\RR^n}  |\widehat{q}(\xi) |^2 |\xi|^{2\alpha-2 +2\lambda}  \int_{N(\xi)}  |\widehat{g}(\eta-\xi)|^2|\eta - \xi|^{n-1-2\lambda} \, d\sigma_\xi(\eta)\, d\xi.
\]
Then the rest of the proof yields $\norm{\widetilde{S}(q)}_{L^2_\alpha} < \infty$. As a consequence Plancherel and \eqref{eq:Q2strct2} are enough to conclude that $\norm{\widetilde{Q}_2(q)}_{W^{\alpha,2}} < \infty$.

The reader may object that in the proof of \Cref{prop:PVSobolev} and \Cref{lemma:PV} we have assumed that $q \in \mathcal S(\RR^n)$. This condition was assumed in order to have \eqref{eq:Kpoint}, but  actually this inequality  holds for the much larger class of potentials satisfying  $\widehat{q} \in  C^1(\RR^n)$, see \cite[Lemma 4.4]{back}. This yields the desired result  since we have assumed that $q$ is compactly supported, which implies $\widehat{q} \in  C^1(\RR^n)$.
\end{proof}

\appendix
\section{} \label{appendixa}

 We give now the proof of  \Cref{lemma:fubini} used in the estimate of the spherical operator $S_r$.
 The case of $r=1$ is proved in \cite{RV}. We prove a more general statement that has been used in \cite{BFPRM2}, for Ewald spheres that depend on two independent parameters instead of one. Let $a,b>0$, we define
 \[ \Phi := \{ (\xi,\eta) \in \RR^n \times \RR^n: |\xi -a \eta| = b|\eta|\},\]
\[ \Gamma_{a,b}(\eta):= \{ \xi \in \RR^n  : |\xi -a \eta| = b|\eta|\},\qquad  N_{a,b}(\xi):= \{ \eta \in \RR^n  : |\xi -a \eta| = b|\eta|\}, \]
and let respectively  be $\sigma_{a,b,\eta}(\xi)$ and $\sigma_{a,b,\xi}(\eta)$   the restriction of the Lebesgue measure to the last two submanifolds of $\RR^n$. In this case $N_{a,b}(\xi)$ is the sphere of center $\frac{a}{(a^2-b^2)}\xi$ and radius $\frac{b}{|a^2-b^2|}|\xi|$.
\begin{lemma} \label{lemma:fubini2}
Let $f\in C^\infty_c(\RR^n \times \RR^n)$ and assume that $a \neq b$  . Then we have that
\[\int_{\RR^n}\int_{\Gamma_{a,b}(\eta)}  f(\eta,\xi)\, d\sigma_{a,b,\eta}(\xi)\, d\eta = \int_{\RR^n} \int_{N_{a,b}(\xi)} f(\eta,\xi) \frac{|\eta|}{|\xi|} \, d\sigma_{a,b,\xi}(\eta) \, d\xi.\]
\end{lemma}
\Cref{lemma:fubini} is just the case $a=1/2$ and $b=r/2$ of the previous statement.
\begin{proof}
The result follows by direct computation using the language of differential forms. We denote respectively by $d\xi = d\xi_1 \wedge \dots  \wedge d\xi_n$ and $d\eta = d\eta_1 \wedge \dots  \wedge d\eta_n$ the volume form of $\RR^n$ in coordinates $(\xi_1,\dots,\xi_2)$ and $(\eta_1,\dots,\eta_2)$. Also, we denote by $\omega_\eta$ the  natural volume $n$-form of the sphere $\Gamma_r(\eta)$ and by $\omega_\xi$ the volume $n$-form of $N_r(\xi)$. Hence $\omega_\eta$ is associated to the measure $\sigma_{a,b,\eta}$ and $\omega_\xi$ to $\sigma_{a,b,\xi}$.

  Since $\Gamma_{a,b}(\eta)$ is an hypersurface, $\omega_\eta$ is just the contraction of its (exterior) unit normal vector field $\nu(\xi)$  with the volume form $d\xi$. Similarly,   $\omega_\xi$ is the contraction with the  unit normal field to $N_{a,b}(\xi)$, $\nu(\eta)$, with the volume form $d\eta$. Since both hypersurfaces are spheres, these vector fields can be computed very easily  in coordinates
\[\nu(\xi)=\frac{1}{b|\eta|}(\xi-a\eta), \hspace{5mm}\text{and} \hspace{5mm}\nu(\eta) = \frac{|a^2-b^2|}{b|\xi|}\left(\eta-\frac{a}{a^2-b^2}\xi\right).\]
 Therefore we can compute the following coordinate expressions,
\begin{align*}
\omega_\eta \wedge d\eta &= \frac{1}{b|\eta|}\sum_{i=1}^n (-1)^{i+1} (\xi_i-a\eta_i) d\xi_1 \wedge \dots \wedge \widehat{d\xi_i}\wedge \dots \wedge d\xi_n \wedge d\eta, \\
\omega_\xi \wedge d\xi &= \frac{|a^2-b^2|}{b|\xi|}\sum_{i=1}^n (-1)^{i+1} \left (\eta_i-\frac{a}{a^2-b^2} \xi_i \right) d\eta_1 \wedge \dots \wedge \widehat{d\eta_i}\wedge \dots \wedge d\eta_n \wedge d\xi, 
\end{align*}
where the notation $\widehat{d\xi_i}$ means that we are omitting the 1-form $d\xi_i$ in the wedge product. $\omega_\eta \wedge d\eta$ and $\omega_\xi \wedge d\xi $ are volume forms on the $\RR^n \times \RR^n$ submanifold $\Phi$. To compare them,  we want to write both forms in coordinates as similarly as possible . This can be achieved  by using the structural relation
\[\sum_{i=1}^{n} \left(  2 (a^2-b^2) \left( \eta_i-\frac{a}{a^2-b^2} \xi_i \right)d\eta_i + 2(\xi_i-a\eta_i)d\xi_i \right) = 0, \]
 obtained just by taking the exterior differential of the function $|\xi-a\eta|^2 - b^2|\eta|^2$, which is constant on $\Phi$ by definition. Assume that we are in the open set given by $(\xi_1-a\eta_1)\neq 0$ (we can choose any of the other possible conditions $(\xi_i-a\eta_i)\neq 0$ without difference). Then we can write
\[ d\xi_1 = \frac{1}{(a\eta_1-\xi_1)} \left( \sum_{i=1}^{n}   (a^2-b^2) \left( \eta_i-\frac{a}{a^2-b^2} \xi_i \right)d\eta_i + \sum_{i=2}^{n} (\xi_i-a\eta_i)d\xi_i \right) .\]
Introducing  this equation in the coordinate expressions of $\omega_\eta \wedge d\eta $ and $\omega_\xi \wedge d\xi$ most products cancel out, and after some computations we obtain that
\begin{align*}
\omega_\eta \wedge d\eta &= \frac{1}{b|\eta|(a\eta_1-\xi_1)}\sum_{i=1}^n -(\xi_i-a\eta_i)^2 d\xi_2 \wedge \dots  \dots \wedge d\xi_n \wedge d\eta \\
&=  -\frac{b|\eta|}{(a\eta_1-\xi_1)} d\xi_2 \wedge \dots  \dots \wedge d\xi_n \wedge d\eta, \\ 
\omega_\xi \wedge d\xi &=  \frac{|a^2-b^2|(a^2-b^2)}{b|\xi|(a\eta_1-\xi_1)}\sum_{i=1}^n (-1)^{n^2-1} \left (\eta_i-\frac{a}{a^2-b^2} \xi_i \right)^2 d\xi_2 \wedge \dots  \dots \wedge d\xi_n \wedge d\eta \\
&= (-1)^{n^2-1} \frac{|a^2-b^2|}{a^2-b^2}\frac{b|\xi|}{(a\eta_1-\xi_1)}   d\xi_2 \wedge \dots  \dots \wedge d\xi_n \wedge d\eta.
\end{align*}
Comparing both expressions we see that except for the sign, both volume forms on $\Phi$ differ by a $|\eta|/|\xi|$ factor. This yields the desired result, returning to the notation with the measures $\sigma_{a,b,\eta}$ and $\sigma_{a,b,\xi}$.
\end{proof}

\begin{proposition} \label{prop:L1Lamda}
Let $f \in \mathcal S'(\RR^n)$. Then we have that
\begin{equation*} 
 \norm{f}_{\Lambda^\alpha} \le   C\norm{\widehat{f}}_{L^1_\alpha} .
\end{equation*}
\end{proposition}
\begin{proof}
Let $m$ be the integer part of $\alpha$, and $\gamma$  any multi-index such that $|\gamma|\le m$. Then we have that
\[\norm{\partial^\gamma f}_{L^\infty} \le C \int_{\RR^n} |\xi^{\gamma}\widehat{f}(\xi)|\,d\xi \le C \int_{\RR^n} \jp{\xi}^{\alpha}|\widehat{f}(\xi)|\,d\xi .\] 
This means that we can reduce the proof to the case $0<\alpha<1$. Expressing $f(x)$ as the inverse Fourier transform of $\widehat{f}(\xi)$ we get
\[\frac{ |f(x+t)-f(x)|}{|t|^{\alpha}} \le C  \int_{\RR^n}  \left| \frac{e^{i\xi\cdot t }-1}{|t|^{\alpha}} \right| |\widehat{f}(\xi)| \,d\xi  .\]
Then is enough to show that 
\[\left| \frac{e^{i\xi\cdot t }-1}{|t|^{\alpha}} \right| \le 2|\xi|^{\alpha}.\]
The previous inequality is immediate for $|\xi| \ge |t|^{-1}$, so we consider  $|\xi| \le |t|^{-1}$. In this case we have $|\xi|| t| \le 1$ which implies
\[|e^{i\xi\cdot t }-1| \le 2|\xi||t| \le 2 |\xi|^{\alpha}|t|^\alpha,\]
and this yields the desired result.
\end{proof}

\begin{proof}[Proof of  $\Cref{lemma:integral2}$]
 Consider $\SP_\rho$ centred in the origin. Assume that $x \neq 0$, and take $\omega \in \SP_\rho$ such that $\omega = x/|x|$. Let $P_\omega = \{x\in \RR^n: x \cdot \omega=0\}$, and let $P(z):= z-(z\cdot\omega)\omega$, be the projection of $z\in \RR^n$ on the plane $P_\omega$. Consider the half sphere comprised between the plane $P_\omega$ and the parallel one that goes trough $x$. The Jacobian of the projection $P$ restricted to $\SP_\rho$ is uniformly bounded in $\rho$ if we exclude a small band of $\rho\varepsilon$ width from it. Let's denote this region by $S_{\rho\varepsilon}$ (the half sphere minus the band). We have that
 \[\int_{\SP_\rho} \frac{1}{|x-y|^{a}\jp{x-y}^{b}} \, d\sigma_\rho(y) \le 2n \int_{S_{\rho\varepsilon}}  \frac{1}{|x-y|^{a}\jp{x-y}^{b}} \, d\sigma_\rho(y),\]
since in the region $S_{\rho\varepsilon}$ the integrand has larger values than in the rest of the sphere (we are in the half which is closer to $x'$, and it is possible to cover generously $\SP^{n-1}$ with $2n$ pieces like $S_{\rho\varepsilon}$). Then we can use the change of variables $z=P(y)$ to integrate in the corresponding region of the plane. Hence, since the integrand is a decreasing function,
\begin{align*}
&\int_{S_{\rho\varepsilon}}  \frac{1}{|x-y|^{a}\jp{x-y}^{b}} \, d\sigma_\rho(y) \le \int_{S_{\rho\varepsilon}}  \frac{1}{|P(y)|^{a}\jp{P(y)}^{b}} \, d\sigma_\rho(y)\\
&\hspace{30mm} \le C \int_{P(S_{\rho\varepsilon})} \frac{1}{|z|^{a}\jp{z}^{b}} \, dz 
 \le C \int_{\RR^{n-1}}\frac{1}{|z|^{a}\jp{z}^{b}} \, dz< \infty ,
\end{align*}
where we have used that $P(x)=0$. 
\end{proof}


\section*{Acknowledgments}

I am very grateful to  my PhD advisors Alberto Ruiz  and Juan Antonio Barceló for their invaluable advice and constant support during the development of this work. 

The author was supported by Spanish government predoctoral grant BES-2015-074055 (project MTM2014-57769-C3-1-P).


  \bibliographystyle{plainurl}
  \bibliography{referencesH_R}

\end{document}